\documentclass[conference]{IEEEtran}
\IEEEoverridecommandlockouts
\usepackage{cite}
\usepackage{graphicx}
\usepackage{mathptmx}
\usepackage{amsmath}
\usepackage{amssymb}
\usepackage{amsthm}
\usepackage{algorithm}
\usepackage{textcomp}
\usepackage{xcolor}
\usepackage{algpseudocode}

\newtheorem{theorem}{Theorem}

\newtheorem{lemma}{Lemma}
\newtheorem{remark}{Remark}

\newcommand{\ve}[1]{\mathbf{#1}}

\def\BibTeX{{\rm B\kern-.05em{\sc i\kern-.025em b}\kern-.08em
    T\kern-.1667em\lower.7ex\hbox{E}\kern-.125emX}}
\begin{document}

\title{On Subsampled Quantile Randomized Kaczmarz\\
\thanks{JH is grateful to and acknowledges support from NSF DMS \#2211318. ER was partially supported by NSF DMS \#2309685.}
}

\author{\IEEEauthorblockN{Jamie Haddock}
\IEEEauthorblockA{\textit{Department of Mathematics} \\
\textit{Harvey Mudd College}\\
Claremont, CA, USA \\
jhaddock@g.hmc.edu}
\and
\IEEEauthorblockN{Anna Ma}
\IEEEauthorblockA{\textit{Department of Mathematics} \\
\textit{University of California, Irvine}\\
Irvine, CA, USA \\
anna.ma@uci.edu}
\and
\IEEEauthorblockN{Elizaveta Rebrova}
\IEEEauthorblockA{\textit{Department of ORFE} \\
\textit{Princeton University}\\
Princeton, NJ, USA \\
elre@princeton.edu}
}

\maketitle

\begin{abstract}
When solving noisy linear systems $A\ve{x} = \ve{b} + \ve{c}$, the theoretical and empirical performance of stochastic iterative methods, such as the Randomized Kaczmarz algorithm, depends on the noise level. However, if there are a small number of highly corrupt measurements, one can instead use quantile-based methods to guarantee convergence to the solution $\ve{x}$ of the system, despite the presence of noise. Such methods require the computation of the entire residual vector, which may not be desirable or even feasible in some cases. In this work, we analyze the sub-sampled quantile Randomized Kaczmarz (sQRK) algorithm for solving large-scale linear systems which utilize a sub-sampled residual to approximate the quantile threshold. We prove that this method converges to the unique solution to the linear system and provide numerical experiments that support our theoretical findings. We additionally remark on the extremely small sample size case and demonstrate the importance of interplay between the choice of quantile and subset size. %\textcolor{brown}{Was not it proposed (but not analysed) in the original quantile paper? At least this is how the main algorithm is stated}
\end{abstract}

%\begin{IEEEkeywords}
%words, things \textcolor{brown}{delete this if %not necessary}
%\end{IEEEkeywords}

\section{Introduction}
With the computational advances of recent decades, the size of datasets regularly analyzed and employed in learning pipelines has skyrocketed.  However, with the growth of these available datasets has come the risk of unperceived yet devastating perturbations and alterations to the input data.  The presence of corruption, outliers, adversarial noise or perturbations can be entirely disruptive to data analysis results or machine learning models~\cite{brendel2017decision}, all while the input data is so large that end users cannot inspect for spurious results~\cite{kurakin2016adversarial,madry2017towards}.  The need for robust methods to corruption, outliers, and adversarial noise has only expanded in recent years and is increasingly the focus across numerous subfields of numerical linear algebra, optimization, statistics, and machine learning.  Furthermore, these methods should be simple to implement, accompanied by strong theoretical guarantees, and flexible to various applications.

Simple iterative methods like those taught in introductory numerical analysis, numerical linear algebra, and numerical optimization courses are prime candidates for corruption robust methods.  The information calculated in-iteration to provide the iterative update can often additionally yield information about the geometry of the problem, the trustworthiness of data, and nearness and existence of a solution.  It has become common to aggregate information across multiple iterations to attempt to mitigate the effect of benign noise~\cite{RKavg20, bach2014adaptivity}. Still, variants using this information to avoid the devastating effects of adversarial corruption in the problem-defining data are newer and less well-understood~\cite{HNRS20, steinerberger2021quantile, pmlr-v108-shah20a, diakonikolas2019sever}.

\section{Problem setup and related literature}

In this work, we consider solving linear problems where a few measurements or components have been corrupted.
Problems in which a small number of untrustworthy data can have a devastating effect on a variable of interest have been considered in ~\cite{awasthi2014power,lai2016agnostic,charikar2017learning,diakonikolas2019robust}.
In particular, linear problems with a small number of outlier measurements have led to an interest in methods for robust linear regression~\cite{rousseeuw1984least,vivsek2006least,NIPS2017_e702e51d}.
Other relevant work includes \emph{min-k loss SGD}~\cite{pmlr-v108-shah20a}, robust SGD~\cite{diakonikolas2019sever,prasad2020robust}, and Byzantine approaches~\cite{blanchard2017machine,alistarh2018byzantine}.

Specifically, our setting here is as follows: consider the setting in which one is given a full rank matrix $A \in \mathbb{R}^{m \times n}$ ($m \gg n$) and a vector $\ve{b} + \ve{c} \in \mathbb{R}^m$, where $A\ve{x} = \ve{b}$ is an overdetermined consistent system with solution $\ve{x}$ and $\ve{c}$ is a sparse corruption vector. Assume the number of corruptions is no more than a fraction $\beta \in (0,1)$ of the total number of measurements, $\|\ve{c}\|_{\ell_0} \le \beta m$ and define the set of corrupted equations to be $C = \text{supp}(\ve{c}) \subset \{ 1,2,\dots,m \}.$ We won't assume any structure or distribution of the corruptions apart from the sparsity.

%\textcolor{brown}{optional: write about the MaxFS approach}

This setting was successfully tackled in recent years via iterative methods that can integrate corruption-avoiding strategies into the iterate's design. One popular and convenient method for solving large-scale linear systems is the Randomized Kaczmarz (RK) method \cite{Kac37:Angenaeherte-Aufloesung, strohmer2009randomized} that solves the system by iterative projections into the individual solution hyperplanes, namely,
\[ \ve x_{k+1}=\ve x_k+\frac{b_{i}-\langle \ve x_k,\ve a_i\rangle}{\|\ve a_i\|^2}\ve a_i. \]
In the consistent overdetermined case, its expected convergence is given by
    $$\mathbb{E}\|\ve{x}_k - \ve{x}\|^2 \le r^k \|\ve{x}_0 - \ve{x}\|^2,$$
    where $r = 1-\sigma^2_{\min}(A)/\|A\|_F^2$.

A simple observation that is of utmost importance for the corrupt setting is that each iteration of the RK algorithm has the step length, $\| \ve{x}_{k+1} - \ve{x}_k\|$, of one of the \emph{residuals}, that is, the distance from the current iterate ${\ve x_k}$ to one of the solution hyperplanes. In \cite{haddock2019randomized}, it was proposed to consider the residual information to identify the ``far apart" corrupted hyperplanes. This method, however, only works when $\ve{x}_k$ is closer to the true solution $\ve{x}_*$ than all other corrupted hyperplanes, and this is achievable only with a very low corruption rate. The next idea was presented in \cite{HNRS20} to consider the residual's stable statistics (quantiles) to identify the hyperplanes that are farther apart from $\ve{x}_k$ and do not use them for the current step only. This method successfully works for the systems with nearly $50\%$ corruption rate (note that $50\%$ is the theoretical threshold for identification of an adversarial corrupted subsystem). It is proven to converge at the per iteration rate of the same order as RK on the uncorrupted system, given that $\beta$ is a small enough constant.

Some follow-up works considered the case when sparse corruptions are mixed with small noise \cite{jarman2021quantilerk, zhang2022quantile}, and gave an alternative approach for the convergence proof \cite{steinerberger2021quantile}. A major practical hurdle of QuantileRK is a very slow exploration phase: to find the quantile estimate; one has to look at all $m$ (recall that $m \gg n$) residuals that are updated at each next iteration ${\bf x}_k$. In \cite{cheng2022block}, the authors propose a way to average over a significant portion of the residual thus using the information obtained on the exploration stage more efficiently. But the question \emph{how the exploration stage can be done faster} remained. In the original QuantileRK paper, the algorithm naturally allows to compute the quantile using only a subset of the residual (\cite{HNRS20}[Algorithm 1]), but the effect of the subset $T$ is never analyzed. Some empirical analysis of using subresidual of the size $0.2m$ is provided in \cite{jarman2021quantilerk}.

In this paper, we give the first explicit analysis of the influence of residual subsampling on the convergence of the QuantileRK method. We give theoretical guarantees (Theorem~\ref{thm:main}) for the case when we subsample a fraction of the total set of equations $\alpha m$ and complement our study with extensive experimental data and the heuristical analysis of the convergence when the sample size is even smaller and have only constant amount of the residuals.

\section{Sub-sampled Quantile Randomized Kaczmarz}

The Sub-sampled Quantile Randomized Kaczmarz (sQRK) algorithm reduces the computational cost of QRK by selecting only a subset of equations to evaluate the residual and estimate the quantile.  For $0 < q < 1$, the $q$-th quantile of a set $S$ is defined as $$\text{$q$-quant}(S) := s \in S \text{ such that } |\{r \in S: r \le s\}| = \lfloor q|S| \rfloor.$$ At every iteration, a subset $\tau_k \subset [m]$ is uniformly selected such that $|\tau_k| \ge \alpha m$ for some $\alpha > 0$. The quantile is then evaluated using the subset $\tau_k$ and a randomly selected equation whose residual is smaller than the quantile is selected to project into. The pseudo-code for sQRK is provided in Algorithm~\ref{alg:sqrk}.

\begin{algorithm}
	\caption{Sub-Sample Quantile Randomized Kaczmarz}\label{alg:sqrk}
	\begin{algorithmic}[1]
		\Procedure{sQRK}{$\ve{A},\hat{\ve{b}}$, $q$, $\alpha$, $N$}
		\State{$\ve{x}_1 = \ve{0}$}
		\For{$k = 1, \ldots, N$}
		\State{Sample $\tau_k \subset [m]$ uniformly such that $|\tau_k| = \lceil \alpha m \rceil $}
		\State{$\gamma_k = \text{$q$-quant}\left(\{| \langle \ve{a}_j,  \ve{x}_{k-1} \rangle - \hat{b}_j | \}_{j \in \tau_k }\right) $}
            \State{$B_k = \{i \in \tau_k : | \langle \ve{a}_i,  \ve{x}_{k-1} \rangle - \hat{b}_i | < \gamma_k \}$}
            \State{Sample $i \sim \text{Unif}\left(B_k\right)$}
		\State{$\ve{x}_{k+1} = \ve{x}_{k} + (\hat{b}_i - \langle \ve{a}_i,  \ve{x}_{k} \rangle ) \ve{a}_i^T$ }
		\EndFor{}
		\Return{$\ve{x}_N$}
		\EndProcedure
	\end{algorithmic}
\end{algorithm}
%\textcolor{red}{Below is still a work in progress -- currently being edited and likely incorrect!}

We define two key sets associated with the algorithm:
\begin{itemize}
    \item $B_k = \{i \in \tau_k : | \langle \ve{a}_i,  \ve{x}_{k-1} \rangle - \hat{b}_i | < q\text{-quant}(S) \}$ is the set of all accepted equations in the sample, and
    \item $S_k:= C \cap B_k$, the set of corrupted equations that can be selected after sampling and applying the quantile threshold.
\end{itemize}
It will be useful to bound the sizes of the sets $B_k$ and $S_k$ and their difference.  First, we note the simple facts that $ B_k \subset \tau_k$ and $S_k \subset  C$. Thus,
\begin{equation}
    |B_k| \ge \alpha q m, \quad \quad |S_k| \le \beta m, \label{eq:BkSksize}
\end{equation}
and so
\begin{equation}
   |B_k\setminus S_k| \ge (\alpha q - \beta)m. \label{eq:BSdiff}
\end{equation}
For the main analysis, we assume that $\alpha q - \beta \gg 0$,  that is, we are guaranteed to have enough uncorrupt equations in any random sample. See Section~\ref{small-samples} for the setting when this does not hold.

\section{Theoretical Guarantees}
In this section, we provide theoretical guarantees for the sQRK algorithm. Theorem~\ref{thm:main} presents our main results, which show that, in expectation, sQRK converges linearly to the solution of the consistent linear system $A\ve{x} = \ve{b}$, despite only having access to $\hat{\ve{b}} = \ve{b} + \ve{c}$. To prove Theorem~\ref{thm:main}, we first prove, in Lemma~\ref{lem:quantilebound}, an upper bound on the quantile of the residual for a subset of selected equations. Using the quantile bound, the expected error is then controlled by conditioning on the event of selecting a corrupt equation in Lemma~\ref{lem:corruptProjection} and the event of selecting a non-corrupt equation in Lemma~\ref{lem:noisyProjection}.

Our main result depends on the factor
$$\sigma_{\alpha, q, \beta, \text{min}}(A) = \min_{S \subset [m], \;|S|/m \ge \alpha q - \beta} \inf_{\ve{x} \not= \ve{0}} \frac{\|A_S \ve{x}\|}{\|\ve{x}\|}.$$
This term will govern the convergence of sQRK conditioned on sampling uncorrupted equations.  With this factor, we can state our main result.

\begin{theorem}
    Let $A \in \mathbb{R}^{m \times n}$ with $m>n$ be a row-normalized, full rank matrix, $\ve{x} \in \mathbb{R}^n$ be fixed, and $\ve{b} = A\ve{x}$. Let $\ve{x}_k$ denote the iterates of Algorithm~\ref{alg:sqrk} applied to matrix $A$ and measurements $\hat{\ve{b}} = \ve{b} + \ve{c}$ where the corruption vector $\ve{c}$ satisfies $\lVert \ve{c} \rVert_{0}=\beta m$. Using quantile $q$, sampling rate $\alpha$, and assuming $\alpha(1-q) > \beta$ and $\alpha q > \beta$,
    if
    \begin{equation}\label{uniform_sn}
    %\frac{\beta m}{\alpha m} \tilde{r}_C < \frac{\sigma_{\alpha,q,\beta,\text{\normalfont{min}}}^2}{\alpha m},
    r_G < \frac{1 - \frac{\beta}{\alpha q}\tilde{r}_C}{1 - \frac{\beta}{\alpha q}}
    \end{equation}
    where
 %   \jh{assuming my blue note in Lemma 1 is correct, I think we can get rid of the $\sqrt{m(1-beta)}$ and $m(1-beta)$ in the numerators of the second two terms...}
    $$\tilde{r}_C = \bigg(1+\frac{2}{\sqrt{\beta m}}\frac{\sigma_{\text{\normalfont{max}}}^2(A)}{\sqrt{m[\alpha(1-q)-\beta]}}+\frac{\sigma_{\text{\normalfont{max}}}^2(A)}{m[\alpha(1-q)-\beta]}\bigg)$$
    and $$r_G = 1 - \frac{\sigma_{\alpha,q,\beta,\text{\normalfont{min}}}^2}{\alpha qm},$$
    then sQRK converges at least linearly in expectation,
    $$\mathbb{E}\|\ve{x}_k - \ve{x}\|^2 \le r^k \|\ve{x}_0 - \ve{x}\|^2$$
    where $r = \left(1 - \frac{\beta}{\alpha q}\right) r_G +  \frac{\beta}{\alpha q} \tilde{r}_C.$
    \label{thm:main}
\end{theorem}

\begin{remark} Note that condition \eqref{uniform_sn} ensures that the rate $r <1$ and the method is indeed convergent. A simple calculation shows that it is equivalent to the condition
$$
\frac{\beta}{\alpha q} + \beta\frac{ \sigma_{\max}^2(A)}{\sigma^2_{\alpha,q, \beta,\min}}\left(\frac{2}{\sqrt{\beta}\sqrt{d}} + \frac{1}{d}\right) < 1
$$
for $d = \alpha(1-q) - \beta$.
See also Figure~\ref{fig:hypothesis_heatmap} that explores joint admissible values for the parameters $\alpha$ and $q$.
\end{remark}
The proofs of Theorem~\ref{thm:main} and its supporting lemmas are similar to those that appear in \cite{steinerberger2021quantile} but are not immediate results of the previous work. In particular, we pay special attention to the impact of the sub-sampling rate $\alpha$ here. Recall that $B_k$ denotes the ``acceptable" set of equations after sampling and using the quantile at iteration $k$ with
\begin{align*}
    \gamma_k(\ve{x}_{k-1}) &= \text{$q$-quant}\left(\{| \langle \ve{a}_j,  \ve{x}_{k-1} \rangle - \hat{b}_j | \}_{j \in \tau_k }\right), \\
    \text{and } B_k &= \{i \in \tau_k : | \langle \ve{a}_i,  \ve{x}_{k-1} \rangle - \hat{b}_i | < \gamma_k \}.
\end{align*}
The set $S_k= C \cap B_k$ this is the set of corrupted equations that can be selected after sampling and using the quantile. If $S_k = \emptyset$, we can be sure that we are projecting onto an uncorrupted equation at iteration $k$.

\begin{lemma}\label{lem:quantilebound}
Assume $|S_k|\geq 1$.
%Suppose that the fraction of corruptions in the system is bounded by $\beta$, $\lVert \ve c^{(k)} \rVert_{\ell^0}\leq\beta m$. \textcolor{brown}{For every $k$? }
Let $\alpha, q > 0$ and assume $\alpha(1-q) > \beta$, then for any arbitrary vector $\ve v\in\mathbb{R}^{n}$ and for all $k$:
\[\gamma_k({\bf v}) \leq\frac{\sigma_{\text{\normalfont{max}}}(A)}{\sqrt{m[\alpha(1-q) - \beta]}}\lVert \ve v - \ve x \rVert. \]
\end{lemma}

\begin{proof}  We begin by considering non-corrupt equations: $i \not\in C$ where
\[\langle \ve a_i, \ve x \rangle = \hat{b}_{i} = b_{i}. \]
%\noindent This is true for each $i\notin C_{k}$, thus if we take the sum over all these elements in $\tau_k$ we have
%\[\sum\limits_{\substack{i\in \tau_k}}| \langle \ve a_i, \ve v \rangle - \langle \ve a_i, \ve x \rangle| \ge \sum\limits_{\substack{i\in \tau_k \\ i\notin C_{k}}}| \langle \ve a_i, \ve v \rangle - \langle \ve a_i, \ve x \rangle| = \sum\limits_{\substack{i \in \tau_k \\ i\notin C_{k}}}|\langle \ve a_i, \ve v \rangle - b_{i}^{(k)}|.\]
\noindent Taking the sum of the squared residuals for non-corrupt equations, we get
\begin{align*}
\sum\limits_{\substack{i=1 \\ i\notin C}}^{m}|\langle \ve a_i,\ve v\rangle - \hat{b}_{i}|^2
& \le \lVert A_{\notin C}\ve v - \ve b_{\notin C} \rVert^2 \\
& = \lVert A_{\notin C}\ve v - A_{\notin C}\ve x \rVert^2  \leq \sigma^2_{\text{\normalfont{max}}}(A)\lVert \ve v - \ve x \rVert^2.
\end{align*}
%\jh{Do we need to pass to the $\ell_1$ norm here?  I think we can just work directly with the $\ell_2$ norm squared.  (P.S.\ I know I wrote this, just don't understand what I was thinking, lol...)}
%Using the $\ell_1 - \ell_2$ norm inequality,
%we get
%\begin{align*}
%\sum\limits_{\substack{i \in \tau_k \\ i\notin C}}|\langle \ve a_i, \ve v \rangle - \hat{b}_{i}| &\leq \sum\limits_{\substack{i=1 \\ i \notin C}}^{m}|\langle \ve a_i, \ve v \rangle - \hat{b}_{i}| \\
%& \leq \sqrt{m(1-\beta)}\sigma_{\text{\normalfont{max}}}(A)\lVert \ve v - \ve x \rVert.
%\end{align*}

At least $\alpha(1-q)m$ of the $\alpha m$ values $\{|\langle \ve v, \ve a_i\rangle - \hat{b}_{i}|\}_{i\in \tau_k}$ are at least $\gamma_k(\ve{v})$ and since $|S_k|\leq|C |\leq \beta m$, at least $\alpha(1-q)m - \beta m$ belong to equations that have not been corrupted. Thus,
\begin{align*} \gamma_k^2(\ve v) m[\alpha(1-q)-\beta] & \leq\sum_{\substack{i \in \tau_k \\ i\notin C}}| \langle \ve a_i,\ve v \rangle - \hat{b}_{i}|^2 \\
& \leq \sigma_{\text{\normalfont{max}}}^2(A)\lVert \ve v - \ve x \rVert^2.
\end{align*}
and therefore
\[ \gamma_k(\ve v) \leq\frac{\sigma_{\text{\normalfont{max}}}(A)}{\sqrt{m[\alpha(1-q)-\beta]}}\lVert \ve v - \ve x \rVert. \]
\end{proof}

\begin{lemma}\label{lem:corruptProjection}
The expected approximation error over the set of acceptable corrupted equations $S_k$ is

\begin{equation*}
\mathbb{E}_{i\in S_{k}}\rVert \ve x_{k+1}-\ve x\lVert^2 \leq r_C \|\ve{x}_k - \ve{x}\|^2,
\end{equation*}
where
\begin{equation}
r_C = 1+\frac{2}{\sqrt{|S_{k}|}}\frac{\sigma_{\text{\normalfont{max}}}^2(A)}{\sqrt{m[\alpha(1-q)-\beta]} } +\frac{\sigma_{\text{\normalfont{max}}}^2(A)}{m [\alpha(1-q)-\beta]}.
\label{eq:corrupterr}
\end{equation}
\end{lemma}
\begin{proof} Recall the definition of $\ve x_{k+1}$ from Algorithm~\ref{alg:sqrk} is
\[ \ve x_{k+1}=\ve x_k+\left(\hat{b}_{i}-\langle \ve x_k,\ve a_i\rangle\right)\ve a_i. \]
The approximation error can be written as
\begin{equation}
 \lVert \ve x_{k+1} -\ve x\rVert^2=\lVert \ve x_k-\ve x\rVert^2 +2\langle \ve x_k-\ve x,\ve v\rangle +\lVert \ve v\rVert^2,
 \label{eq:corrupt-sum}
 \end{equation}
where
\[ \ve v=(\hat{b}_{i}-\langle \ve x_k,\ve a_i\rangle) \ve a_i .\]
We proceed by bounding the last two terms of \eqref{eq:corrupt-sum}.
For $i\in S_{k}$, the term $\lVert \ve v \rVert^2$ is uniformly small since
\begin{align}
{\lVert \ve v \rVert^2} &= \lVert (\hat{b}_{i}-\langle \ve x_k,\ve a_i\rangle) \ve a_i \rVert^2 = |\hat{b}_{i} -\langle \ve x_k,\ve a_i\rangle |^2  \nonumber\\
&\leq \gamma_k^2(\ve x_k) \nonumber\\
&\leq \frac{\sigma^2_{\text{\normalfont{max}}}(A)}{m[\alpha(1-q) - \beta]}\lVert \ve x_k - \ve x \rVert^2, \label{eq:corrupt-term3}
\end{align}
where the second to last inequality follows from $S_k \subseteq B_k$ and the last inequality follows from Lemma~\ref{lem:quantilebound}.

It remains to bound $\mathbb{E}_{i\in S_{k}}2\langle \ve x_k - \ve x,\ve v\rangle$. We begin by showing
\begin{align*}
\mathbb{E}_{i\in S_{k}}2 \langle \ve x_k-\ve x,\ve v\rangle
&= \frac{2}{|S_{k}|}\sum_{i\in S_{k}}\langle \ve x_k-\ve x,(\hat{b}_{i}-\langle \ve x_k,\ve a_i\rangle)\ve a_i\rangle \\
&= \frac{2}{|S_{k}|}\sum_{i\in S_{k}}(\hat{b}_{i}-\langle \ve x_k,\ve a_i\rangle)\langle \ve x_k- \ve x, \ve a_i \rangle
\end{align*}
The Cauchy-Schwarz inequality yields
\begin{align*}
    \sum_{i\in S_{k}}(\hat{b}_{i}-&\langle \ve x_k,\ve a_i\rangle)\langle \ve x_k-\ve x,\ve a_i\rangle \\
    &\leq \Bigg(\sum_{i\in S_{k}}(\hat{b}_{i}-\langle \ve x_k,\ve a_i\rangle)^2 \sum_{i\in S_{k}}\langle \ve x_k-\ve x,\ve a_i\rangle^2\Bigg)^{\frac{1}{2}},
\end{align*}
and thus
\begin{align*}
\mathbb{E}_{i\in S_{k}}&2\langle \ve x_k-\ve x, \ve v\rangle \\&\leq \frac{2}{|S_{k}|}\Bigg(\sum_{i\in S_{k}}(\hat{b}_{i}-\langle \ve x_k,\ve a_i\rangle)^2 \sum_{i\in S_{k}}\langle \ve x_k-\ve x,\ve a_i\rangle^2\Bigg)^{\frac{1}{2}} \\
&\leq \frac{2}{|S_{k}| } \sqrt{|S_{k}|}\frac{ \sigma_{\text{\normalfont{max}}}(A)}{\sqrt{m[\alpha(1-q)-\beta]}} \lVert \ve x_k - \ve x \rVert\\
& \quad \quad \quad \quad \quad \quad \quad \quad \quad \quad \times   \bigg(\sum_{i\in S_{k}}\langle \ve x_k-\ve x,\ve a_i\rangle^2\bigg)^{\frac{1}{2}}.
\end{align*}

 At this point, we estimate
\begin{align*}
\sum_{i\in S_{k}}\langle \ve x_k-\ve x,\ve a_i\rangle^2 &\leq \sum_{i=1}^{m} \langle \ve x_k-\ve x,\ve a_i\rangle^2 \\
& = \lVert A(\ve x_k-\ve x)\rVert^2\leq\sigma_{\text{\normalfont{max}}}^2(A)\lVert \ve x_k-\ve x\rVert^2,
\end{align*}
and hence
\begin{align}
\mathbb{E}_{i\in S_{k}} 2\langle \ve x_k-\ve x,v\rangle&\leq\frac{2}{\sqrt{|S_{k}|}}\frac{\sigma_{\text{\normalfont{max}}}^2(A) }{\sqrt{m[\alpha(1-q)-\beta]}} \lVert \ve x_k-\ve x\rVert^2 \label{eq:corrupt-term2}.
\end{align}
Taking the expectation over $i \in S_k$ in \eqref{eq:corrupt-sum} and using the bounds \eqref{eq:corrupt-term3} and \eqref{eq:corrupt-term2} obtains the final result:
\begin{equation*}
\mathbb{E}_{i\in S_{k}}\rVert \ve x_{k+1}-\ve x\lVert^2\leq r_C\lVert \ve x_k-\ve x\rVert^2,
\end{equation*}
where $r_C$ is as defined in \eqref{eq:corrupterr}.
\end{proof}

Now we consider the subset $B_k\setminus S_{k}$ and the event in which we project onto an uncorrupted equation.

\begin{lemma} \label{lem:noisyProjection}
Let $A \in \mathbb{R}^{m \times n}$ with $m>n$ be a row-normalized, full rank matrix, $\ve{x} \in \mathbb{R}^n$ be fixed, and $\ve{b} = A\ve{x}$. Let $\ve{x}_k$ denote the iterates of Algorithm~\ref{alg:sqrk} applied to matrix $A$ and measurements $\hat{\ve{b}} = \ve{b} + \ve{c}$ where the corruption vector $\ve{c}$ satisfies $\lVert \ve{c} \rVert_{0}=\beta m$. Using quantile $q$, sampling rate $\alpha$, and assuming $\alpha q > \beta$, and conditioning on the case that the sampled row index is uncorrupted yields at least linear convergence in expectation with% given by Theorem~\ref{thm:steinerbergerQRKwNoise},
\[ \mathbb{E}_{i\in B_k \setminus S_{k}}\lVert \ve x_{k+1}-\ve x\rVert^2\leq r_G \lVert \ve x_k-\ve x\rVert^2,\]
where
\begin{equation}
    r_G = 1-\frac{\sigma_{\alpha,q,\beta,\text{\normalfont{min}}}^2(A)}{\alpha qm}.
\end{equation}
\end{lemma}

The proof of this lemma follows directly from the proof of \cite[Lemma 3]{steinerberger2021quantile} where $B$ and $S$ are substituted by $B_k$ and $S_k$, and we use the estimate $\|A_{B_k\setminus S_k}\|_F^2 \le \alpha q m$.

\begin{proof}[Proof of Theorem~\ref{thm:main}]
By the Law of Total Expectation and noting that $\mathbb{P}(i \not\in B_k) = 0$, we have
\begin{align*}
    \mathbb{E} \|\ve{x}_{k+1} - \ve{x}\|^2 &= \mathbb{P}(i \in B_k\setminus S_k) \mathbb{E}_{i \in B_k\setminus S_k} \|\ve{x}_{k+1} - \ve{x}\|^2 \\
    & \quad \quad+ \mathbb{P}(i \in S_k) \mathbb{E}_{i \in S_k} \|\ve{x}_{k+1} - \ve{x}\|^2 \\
    %& \quad \quad+ \mathbb{P}(i \not\in B_k) \mathbb{E}_{i \not\in B_k} \|\ve{x}_{k+1} - \ve{x}\|^2\\
    &\le \frac{\alpha qm - |S_k|}{\alpha qm} r_G \lVert \ve x_k-\ve x\rVert^2 \\
    & \quad \quad \quad + \frac{|S_k|}{\alpha qm} r_C \lVert \ve x_k-\ve x\rVert^2 \\ %+ \textcolor{purple}{1-q?}  \|\ve{x}_{k} - \ve{x}\|^2 \\
    &\le \Bigg[r_G + \frac{\beta}{\alpha q} (\tilde{r}_C - r_G) \Bigg] \|\ve{x}_k - \ve{x}\|^2
    \\&= r \|\ve{x}_k - \ve{x}\|^2,
\end{align*}
%\textcolor{purple}{Dont we have $1 - q + qr_G + \frac{\beta}{\alpha}(r_C - r_G) =
 %1 - \frac{\sigma^2}{\alpha m} + \frac{\beta}{\alpha}(r_C - r_G) $ and this is a bit better than current $r$ tilde, instead of +1 we would have + fraction with $\sigma^2$ This would help track dependence on alpha}
where we have used the fact that $r_C - r_G \ge 0$ and $0 \le |S_k| \le \beta m$ in the second inequality to bound $\frac{|S_k|}{\alpha qm} (r_C - r_G)$ and to conclude that $\frac{|S_k|}{\alpha qm} r_C \le \frac{\beta}{\alpha q} \tilde{r}_C$.

%\am{For $r_c \leq \tilde{r}_c$, are we using $|S_k| \leq \beta m$ to bound $\frac{2}{\sqrt{|S_k|}}$: I think this wouldn't give us an upper bound but a lower bound instead: $\frac{2}{\sqrt{|S_k|}} \geq \frac{2}{\sqrt{\beta m}}$? Can we take $|S_k| \geq 1$ in the assumption and have $r_c \leq \tilde{r}_c$ where
%$$\tilde{r}_c = \bigg(1+2\frac{\sigma_{\text{\normalfont{max}}}^2(A)}{\sqrt{m[\alpha(1-q)-\beta]}}+\frac{\sigma_{\text{\normalfont{max}}}^2(A)}{m[\alpha(1-q)-\beta]}\bigg)$$}
\end{proof}

\section{Numerical Experiments}

The experiments presented in this section were performed in MATLAB R2021b on a MacBook Pro 2019 with a 2.3 GHz 8-Core Intel Core i9 processor and 16 GB 2667 MHz DDR4 RAM.

\subsection{Checking assumptions}

In Figure~\ref{fig:hypothesis_heatmap}, we plot a heatmap indicating which hyperparameter sets $(q,\alpha)$ satisfy the assumptions of Theorem~\ref{thm:main}, $\alpha(1-q) > \beta$, $\alpha q > \beta$, and $$ r_G < \frac{1 - \frac{\beta}{\alpha q}\tilde{r}_C}{1 - \frac{\beta}{\alpha q}}.$$  Here, our system is defined by a matrix $A \in \mathbb{R}^{50000 \times 100}$ which is generated with i.i.d.\ $\mathcal{N}(0,1)$ entries and then row-normalized.  We approximate $\sigma_{\alpha,q,\beta,\min}(A)$ by taking 100 uniform random samples of index subsets of size at least $\lceil (\alpha q - \beta)m \rceil$ and recording the minimum singular value encountered in these submatrices.  The heatmap indicates ``T" (true) if all three assumption hold for the indicated $(q,\alpha)$ pair and value of $\beta$ and ``F" (false) otherwise.  We provide heatmaps for $\beta = 10^{-5}$ (upper left), $\beta = 10^{-4}$ (upper right), $\beta = 10^{-3}$ (lower left), and $\beta = 10^{-2}$ (lower right).  We note that $|C| = 0$ for $\beta = 10^{-5}$, $|C| = 5$ for $\beta = 10^{-4}$, $|C| = 50$ for $\beta = 10^{-3}$, and $|C| = 500$ for $\beta = 10^{-2}$.  As expected, the region of pairs $(q, \alpha)$ satisfying the assumptions is larger for smaller corruption rate~$\beta$.

\begin{figure}
    \includegraphics[width=0.5\columnwidth]{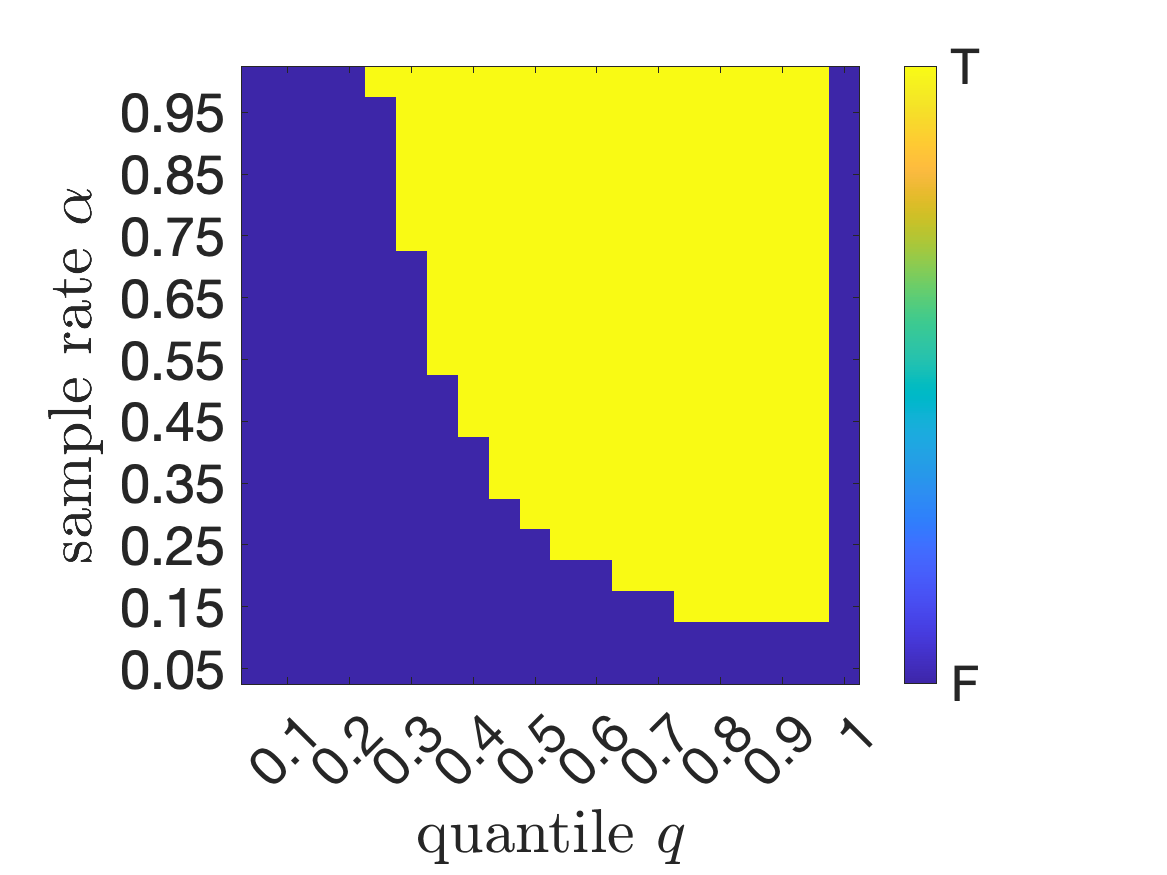}%
    \includegraphics[width=0.5\columnwidth]{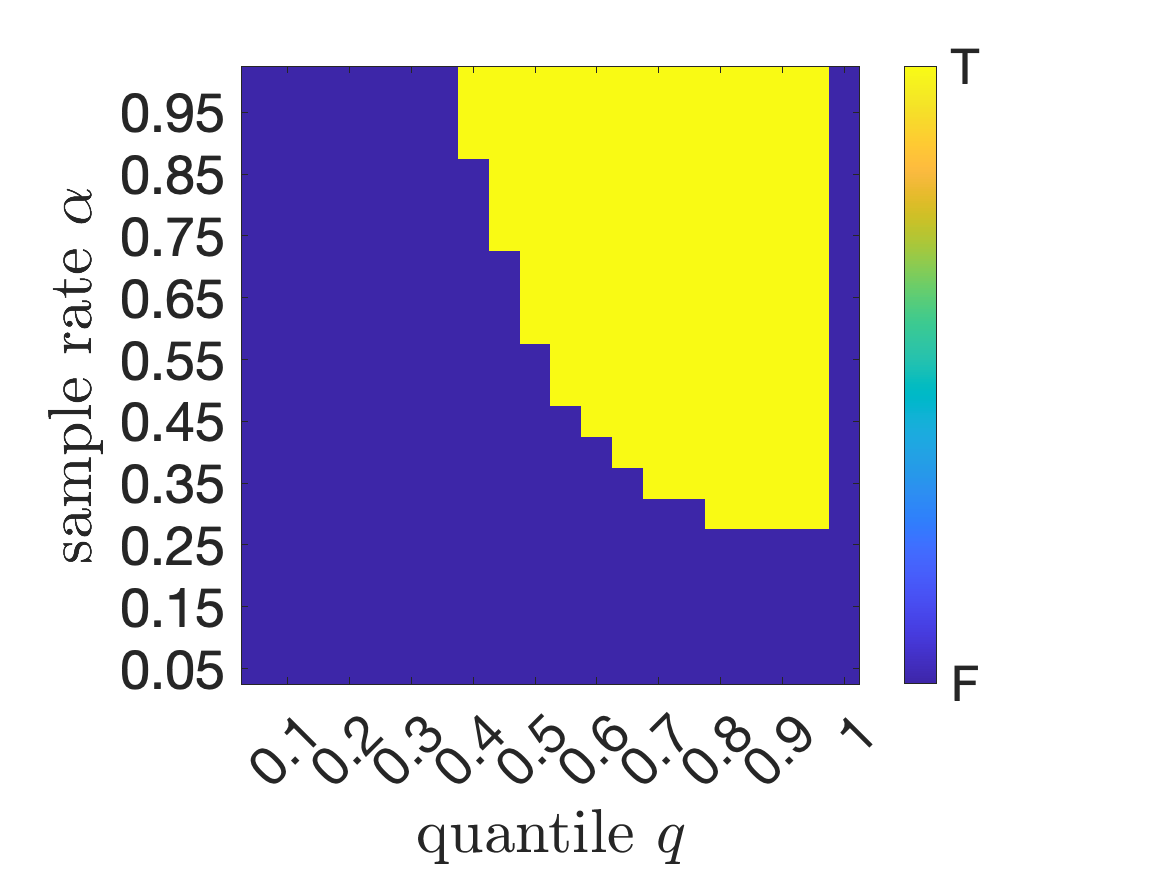}
    \includegraphics[width=0.5\columnwidth]{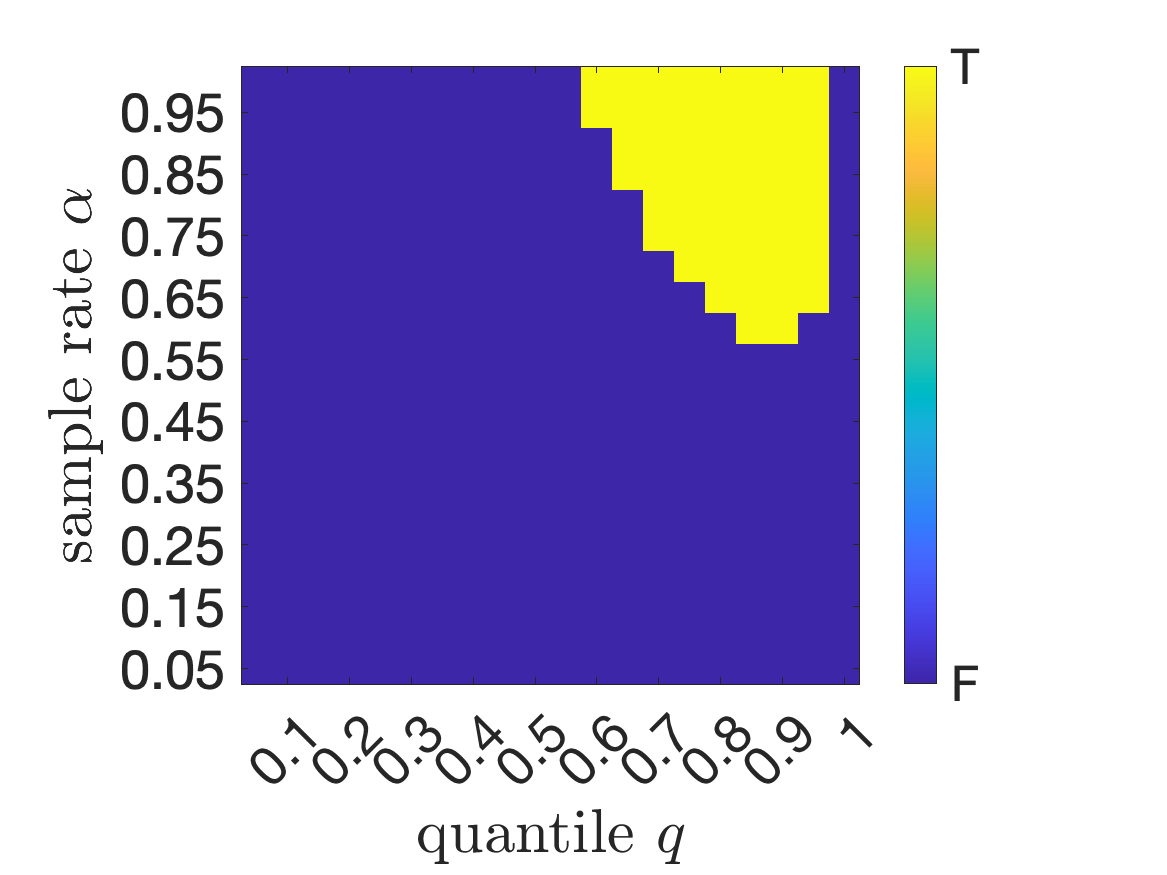}%
    \includegraphics[width=0.415\columnwidth]{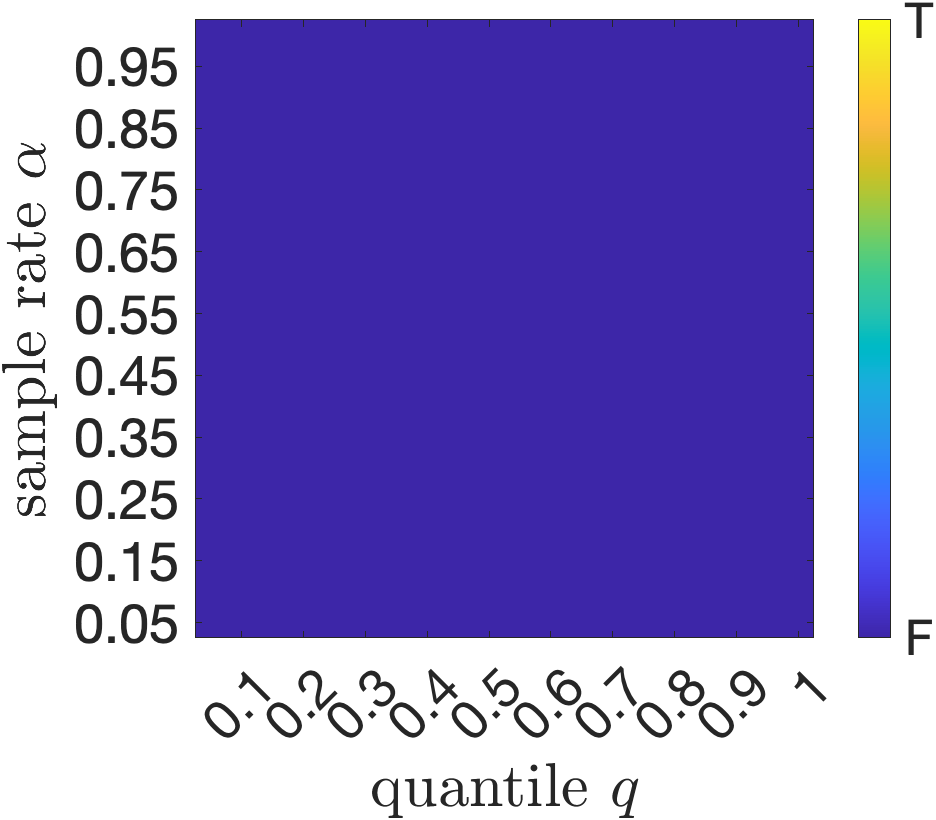}
    \caption{Heatmap indicating which hyperparameter sets $(q,\alpha)$ satisfy the assumptions of Theorem~\ref{thm:main}, $\alpha(1-q) > \beta$, $\alpha q > \beta$, and \eqref{uniform_sn}, for $\beta = 10^{-5}$ (upper left), $\beta = 10^{-4}$ (upper right), $\beta = 10^{-3}$ (lower left), and $\beta = 10^{-2}$ (lower right) for a system defined by row-normalized $A \in \mathbb{R}^{50000 \times 100}$ with i.i.d.\ $\mathcal{N}(0,1)$ entries.}\label{fig:hypothesis_heatmap}
\end{figure}

\subsection{Empirical convergence}

In Figures~\ref{fig:convergence_1e-05}-\ref{fig:convergence_1e-02}, we plot the empirical convergence of ten trials of sQRK with $q = 0.9$ and a variety of $\beta$ and $\alpha$ values.  In each figure, we plot the empirical convergence of ten independent trials of 1000 iterations of sQRK with respect to wall clock time on the left and with respect to iteration $k$ on the right.  In the error versus time plots on the left, we plot a cloud indicating the errors of the 10 independent trials (brighter color indicates more trial errors were below) and lines indicating the mean error of the ten trials.  In the error versus iteration plots on the right, we only plot the mean error over the ten trials as the errors are highly similar for different values of $\alpha$.  We additionally plot the bounds given by Theorem~\ref{thm:main} if they are decreasing.

For Figures~\ref{fig:convergence_1e-05}-\ref{fig:convergence_1e-02}, we generate a single system defined by $A \in \mathbb{R}^{50000 \times 100}$ and $\ve{b} = \ve{0} \in \mathbb{R}^{50000}$.  We generate $A$ with i.i.d.\ $\mathcal{N}(0,1)$ entries, then row-normalize it.  In each figure, we generate $\ve{c}$ to have $\lfloor \beta m \rfloor$ nonzero entries each with value ten.  Thus $\hat{\ve{b}} = \ve{c}$ has  $\lfloor \beta m \rfloor$ nonzero entries each with value ten.  We approximate $\sigma_{\alpha,q,\beta,\min}(A)$ by recording the minimum singular value encountered in each of the $B_k\setminus C$ submatrices encountered during the ten trials of 1000 iterations of sQRK.

In Figure~\ref{fig:convergence_1e-05}, we plot the empirical convergence of ten trials of 1000 iterations of sQRK with $q = 0.9$, $\beta = 10^{-5}$, and $\alpha = 1, 0.5,$ and $0.15$ values.  We plot the empirical convergence of ten independent trials of 1000 iterations of sQRK with respect to wall clock time on the left (brighter color of cloud indicates more trial errors were below) and the mean error for each $\alpha$, and we plot the mean error for each $\alpha$ with respect to iteration $k$ on the right.  We additionally plot the bounds given by Theorem~\ref{thm:main} for each $\alpha$ (all were decreasing).

\begin{figure}
    \includegraphics[width=0.48\columnwidth]{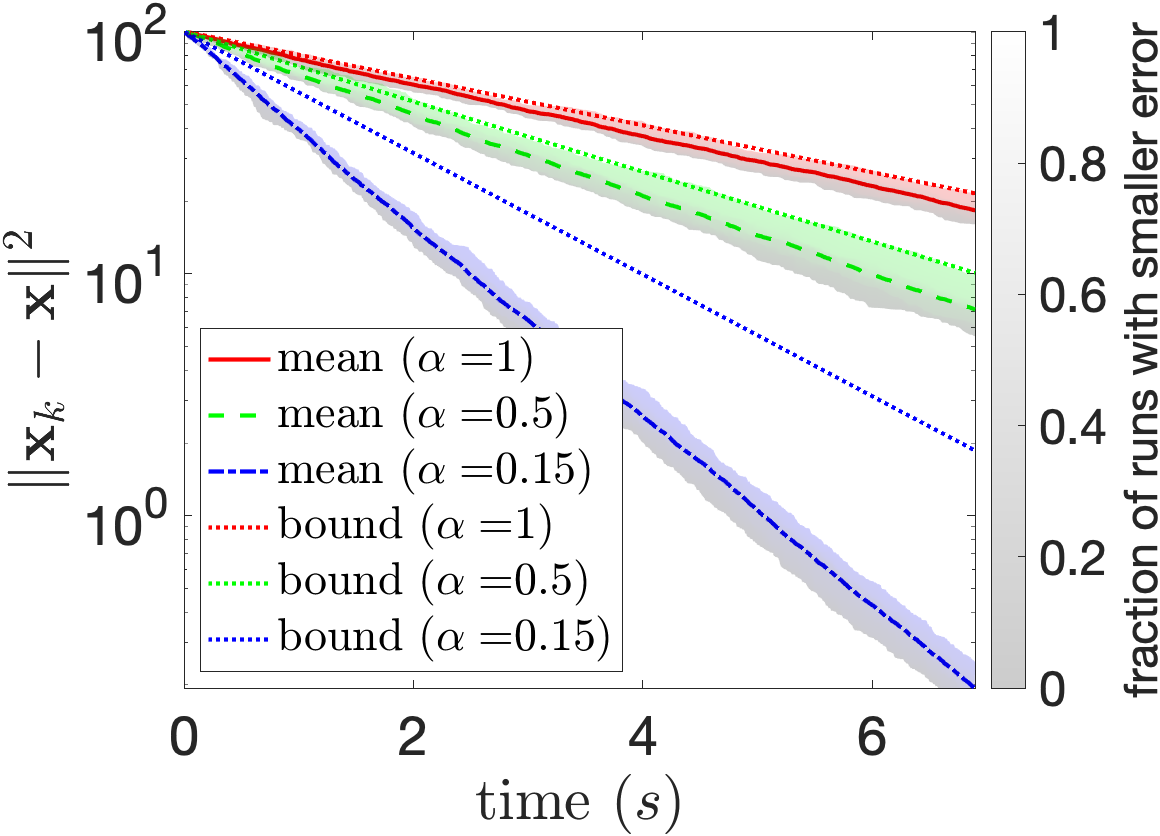}\hfill%
    \includegraphics[width=0.48\columnwidth]{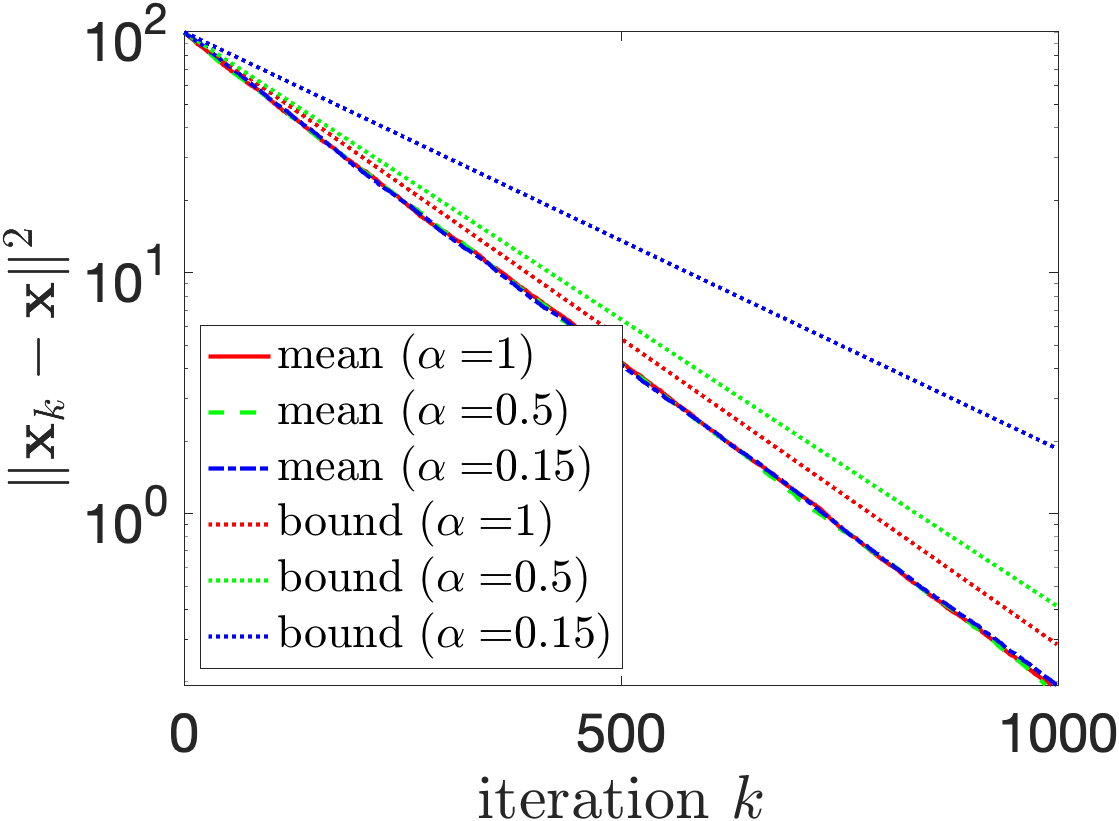}
    \caption{Empirical convergence with respect to time (left) and iterations (right) for ten trials of sQRK with $\alpha = 1, 0.5,$ or $0.15$ on system defined by row-normalized $A \in \mathbb{R}^{50000 \times 100}$ and $\hat{\ve{b}} \in \mathbb{R}^{50000}$ with $\lfloor \beta m \rfloor$ corrupted entries where $\beta = 10^{-5}$.  We additionally plot the bounds provided by Theorem~\ref{thm:main} in dotted lines if they are decreasing.}\label{fig:convergence_1e-05}
\end{figure}

In Figure~\ref{fig:convergence_1e-04}, we plot the empirical convergence of ten trials of 1000 iterations of sQRK with $q = 0.9$, $\beta = 10^{-4}$, and $\alpha = 1, 0.5,$ and $0.15$ values.  We plot the empirical convergence of ten independent trials of 1000 iterations of sQRK with respect to wall clock time on the left (brighter color of cloud indicates more trial errors were below) and the mean error for each $\alpha$, and we plot the mean error for each $\alpha$ with respect to iteration $k$ on the right.  We additionally plot the bounds given by Theorem~\ref{thm:main} for each $\alpha$ (all were decreasing).

\begin{figure}
    \includegraphics[width=0.48\columnwidth]{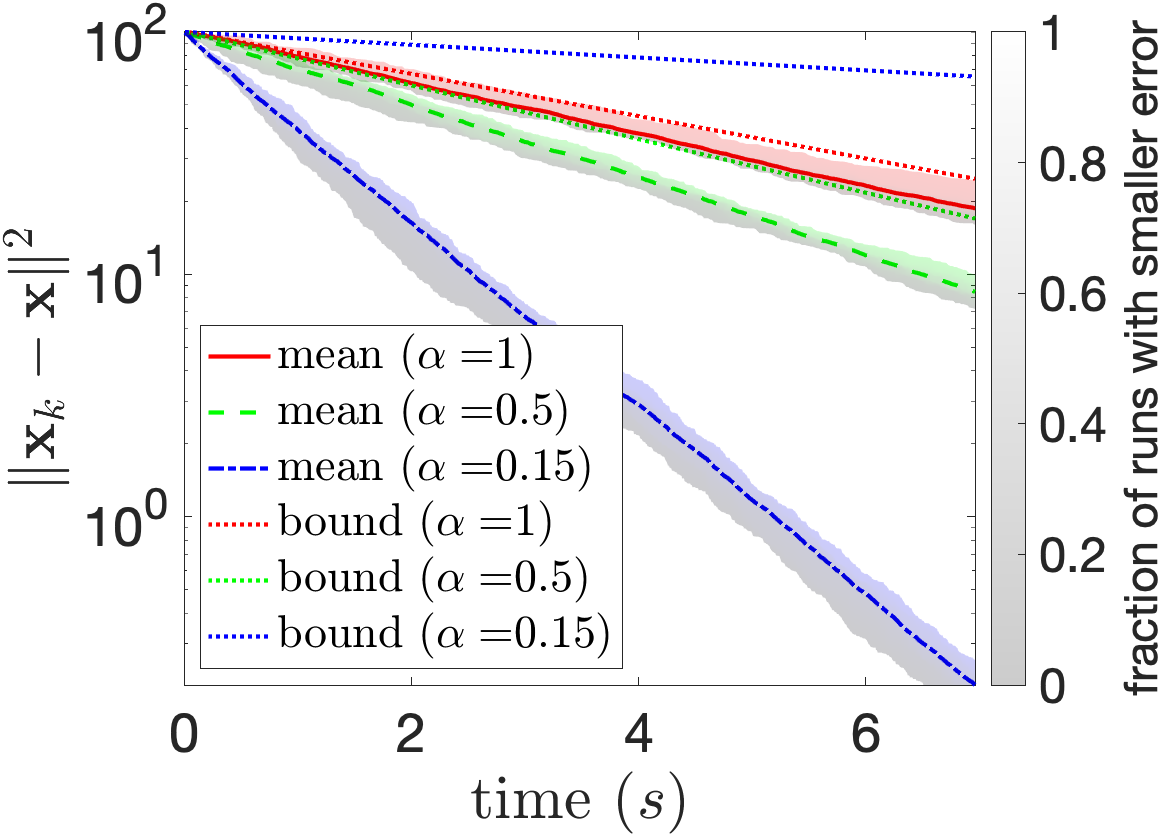}\hfill%
    \includegraphics[width=0.48\columnwidth]{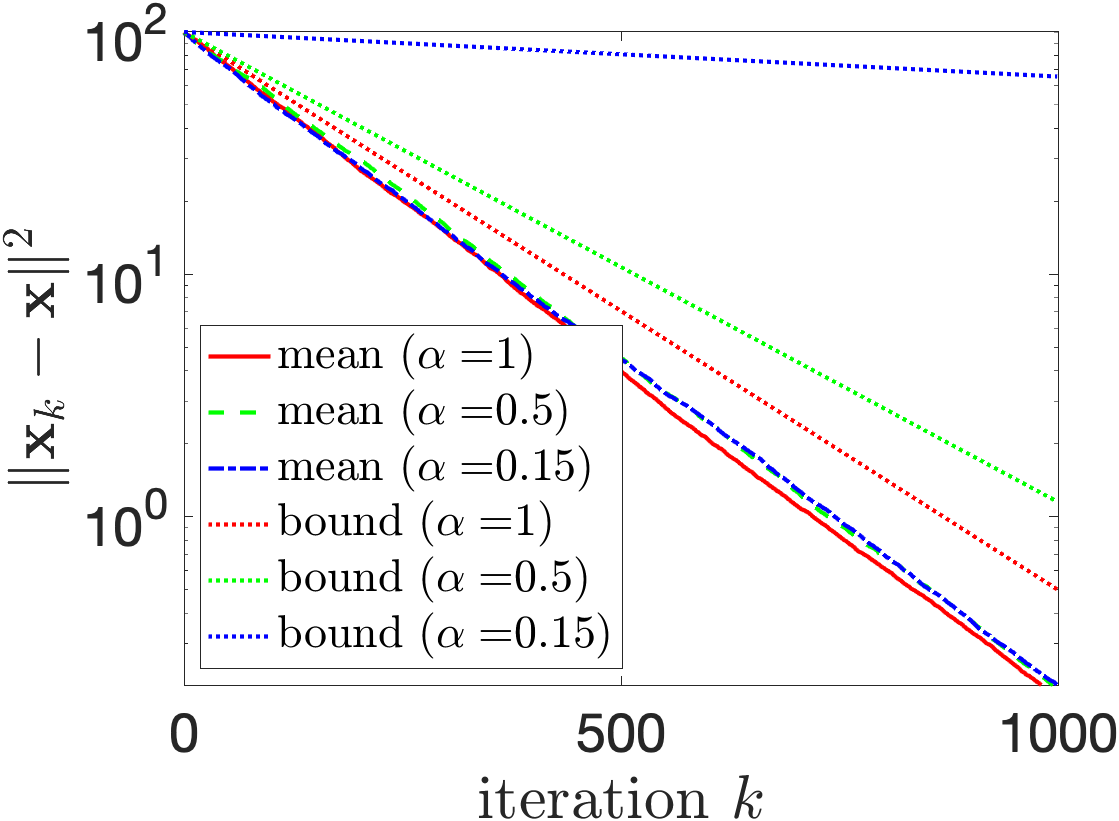}
    \caption{Empirical convergence with respect to time (left) and iterations (right) for ten trials of sQRK with $\alpha = 1, 0.5,$ or $0.15$ on system defined by row-normalized $A \in \mathbb{R}^{50000 \times 100}$ and $\hat{\ve{b}} \in \mathbb{R}^{50000}$ with $\lfloor \beta m \rfloor$ corrupted entries where $\beta = 10^{-4}$.  We additionally plot the bounds provided by Theorem~\ref{thm:main} in dotted lines if they are decreasing.}\label{fig:convergence_1e-04}
\end{figure}

In Figure~\ref{fig:convergence_1e-03}, we plot the empirical convergence of ten trials of 1000 iterations of sQRK with $q = 0.9$, $\beta = 10^{-3}$, and $\alpha = 1, 0.5,$ and $0.15$ values.  We plot the empirical convergence of ten independent trials of 1000 iterations of sQRK with respect to wall clock time on the left (brighter color of cloud indicates more trial errors were below) and the mean error for each $\alpha$, and we plot the mean error for each $\alpha$ with respect to iteration $k$ on the right.  We additionally plot the bounds given by Theorem~\ref{thm:main} for each $\alpha = 1$ and $\alpha = 0.5$.

\begin{figure}
    \includegraphics[width=0.48\columnwidth]{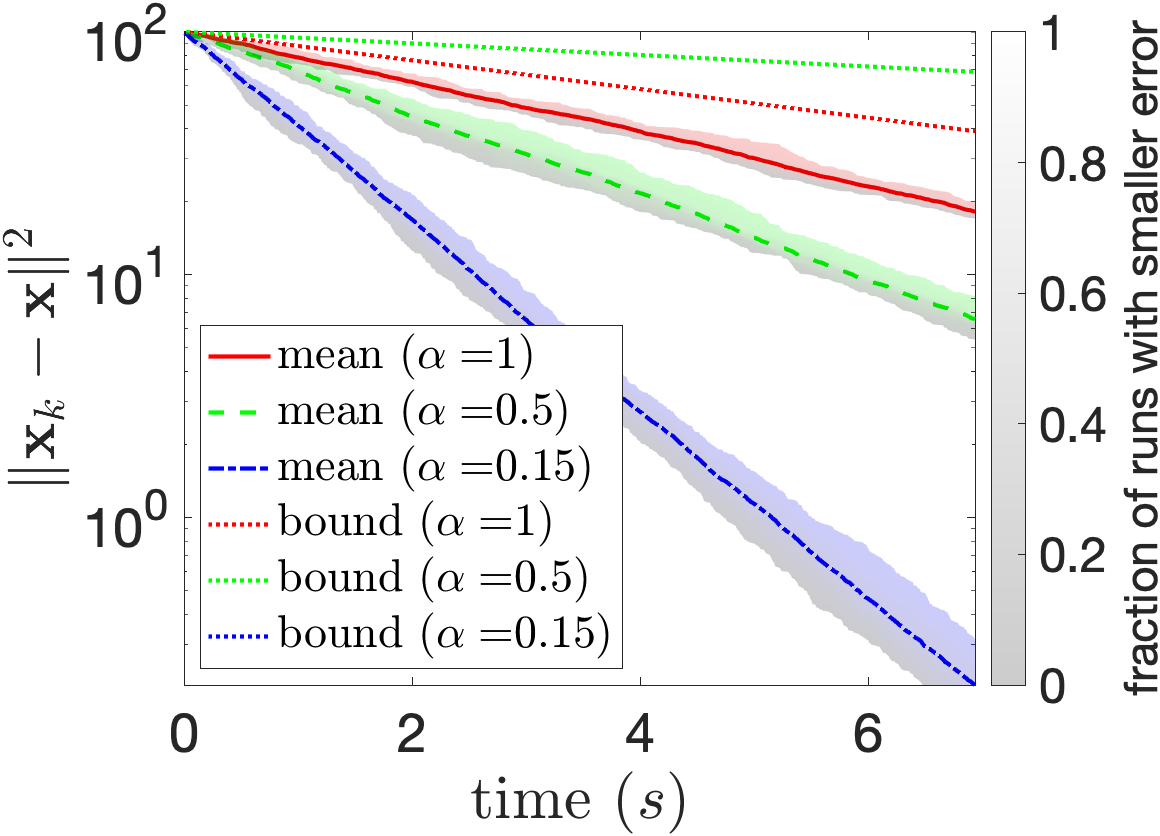}\hfill%
    \includegraphics[width=0.48\columnwidth]{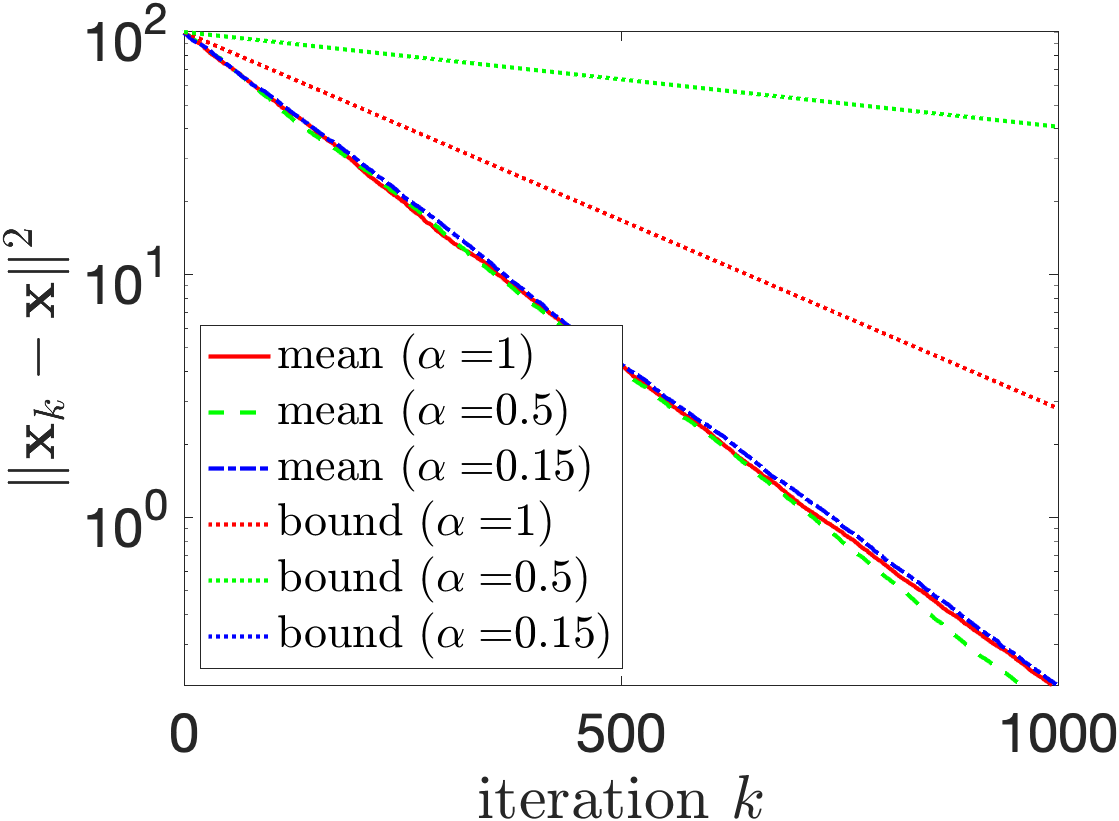}
    \caption{Empirical convergence with respect to time (left) and iterations (right) for ten trials of sQRK with $\alpha = 1, 0.5,$ or $0.15$ on system defined by row-normalized $A \in \mathbb{R}^{50000 \times 100}$ and $\hat{\ve{b}} \in \mathbb{R}^{50000}$ with $\lfloor \beta m \rfloor$ corrupted entries where $\beta = 10^{-3}$.  We additionally plot the bounds provided by Theorem~\ref{thm:main} in dotted lines if they are decreasing.}\label{fig:convergence_1e-03}
\end{figure}

In Figure~\ref{fig:convergence_1e-03}, we plot the empirical convergence of ten trials of 1000 iterations of sQRK with $q = 0.9$, $\beta = 10^{-2}$, and $\alpha = 1, 0.5,$ and $0.15$ values.  We plot the empirical convergence of ten independent trials of 1000 iterations of sQRK with respect to wall clock time on the left (brighter color of cloud indicates more trial errors were below) and the mean error for each $\alpha$, and we plot the mean error for each $\alpha$ with respect to iteration $k$ on the right.  Note that none of the bounds given by Theorem~\ref{thm:main} for any $\alpha$ were decreasing so they do not appear.

\begin{figure}
    \includegraphics[width=0.48\columnwidth]{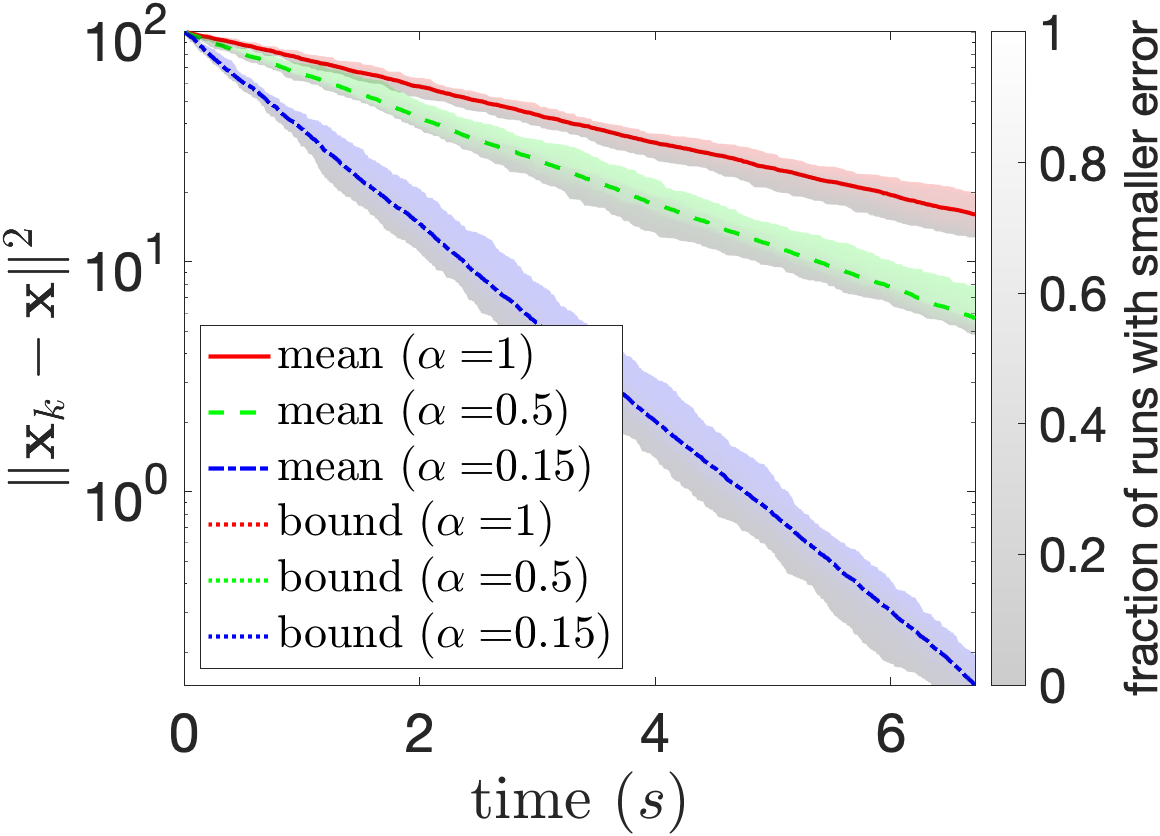}\hfill%
    \includegraphics[width=0.48\columnwidth]{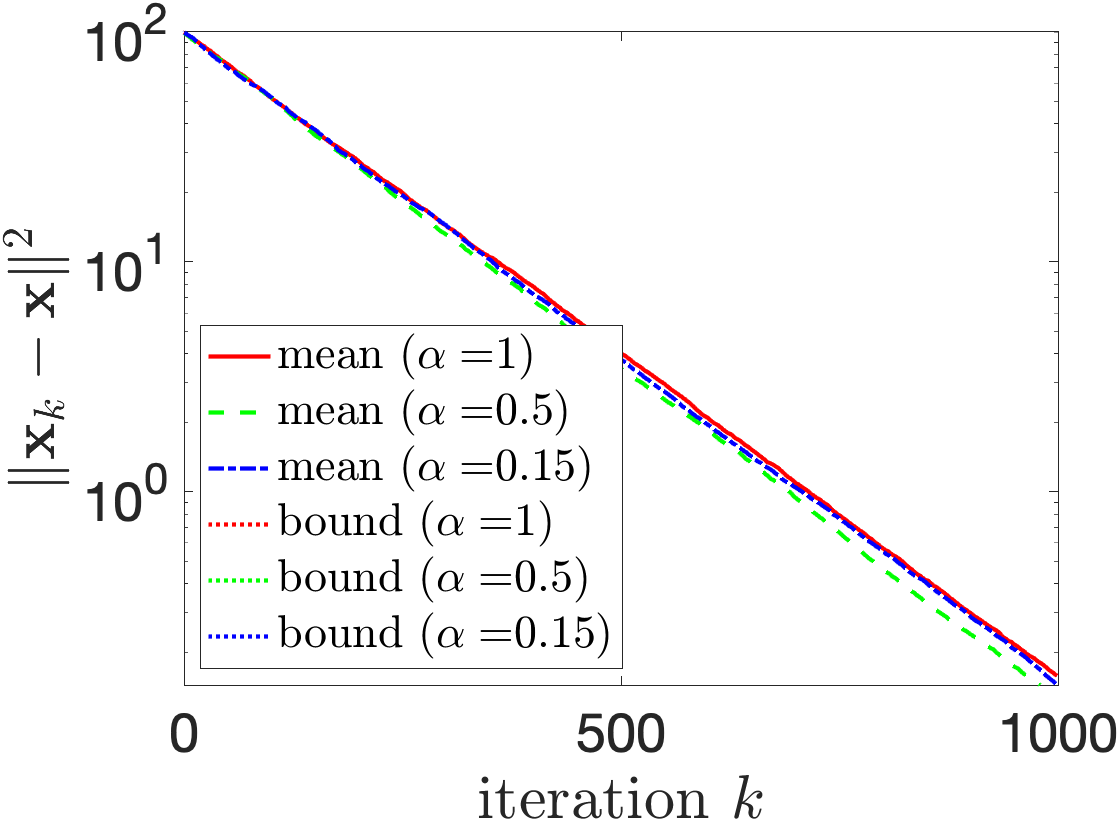}
    \caption{Empirical convergence with respect to time (left) and iterations (right) for ten trials of sQRK with $\alpha = 1, 0.5,$ or $0.15$ on system defined by row-normalized $A \in \mathbb{R}^{50000 \times 100}$ and $\hat{\ve{b}} \in \mathbb{R}^{50000}$ with $\lfloor \beta m \rfloor$ corrupted entries where $\beta = 10^{-2}$.  We additionally plot the bounds provided by Theorem~\ref{thm:main} in dotted lines if they are decreasing.}\label{fig:convergence_1e-02}
\end{figure}

\subsection{Varying $q$}

In Figures~\ref{fig:convergence_qs_1e-03} and~\ref{fig:convergence_qs_1e-04}, we plot the empirical convergence of ten trials of sQRK with and a variety of $\beta$, $\alpha$, and $q$ values.  In each figure, we plot the empirical convergence of ten independent trials of 1000 iterations of sQRK with respect to wall clock time on the left and with respect to iteration $k$ on the right.  In the error versus time plots on the left, we plot a cloud indicating the errors of the 10 independent trials (brighter color indicates more trial errors were below) and lines indicating the mean error of the ten trials.  In the error versus iteration plots on the right, we only plot the mean error over the ten trials as the errors are highly similar for different values of $\alpha$.  We plot the bounds given by Theorem~\ref{thm:main} if they decrease.

For Figures~\ref{fig:convergence_qs_1e-03} and~\ref{fig:convergence_qs_1e-04}, we generate a single system defined by $A \in \mathbb{R}^{50000 \times 100}$ and $\ve{b} = \ve{0} \in \mathbb{R}^{50000}$.  We generate $A$ with i.i.d.\ $\mathcal{N}(0,1)$ entries, then row-normalize it.  In each figure, we generate $\ve{c}$ to have $\lfloor \beta m \rfloor$ nonzero entries each with value ten.  Thus $\hat{\ve{b}} = \ve{c}$ has  $\lfloor \beta m \rfloor$ nonzero entries each with value ten.  We approximate $\sigma_{\alpha,q,\beta,\min}(A)$ by recording the minimum singular value encountered in each of the $B_k\setminus C$ submatrices encountered during the ten trials of 1000 iterations of sQRK.

In Figure~\ref{fig:convergence_qs_1e-03}, we plot the empirical convergence of ten trials of 1000 iterations of sQRK with $\beta = 10^{-3}$, $\alpha = 1$ (top), $\alpha = 0.5$ (middle), and $\alpha = 0.15$ (bottom), and $q = 0.5, 0.7,$ and $0.9$.  We plot the empirical convergence of ten independent trials of 1000 iterations of sQRK with respect to wall clock time on the left (brighter color of cloud indicates more trial errors were below) and the mean error for each $\alpha$. We plot the mean error for each $\alpha$ with respect to iteration $k$ on the right.  We plot the bounds given by Theorem~\ref{thm:main} that were decreasing.

\begin{figure}
    \includegraphics[width=0.48\columnwidth]{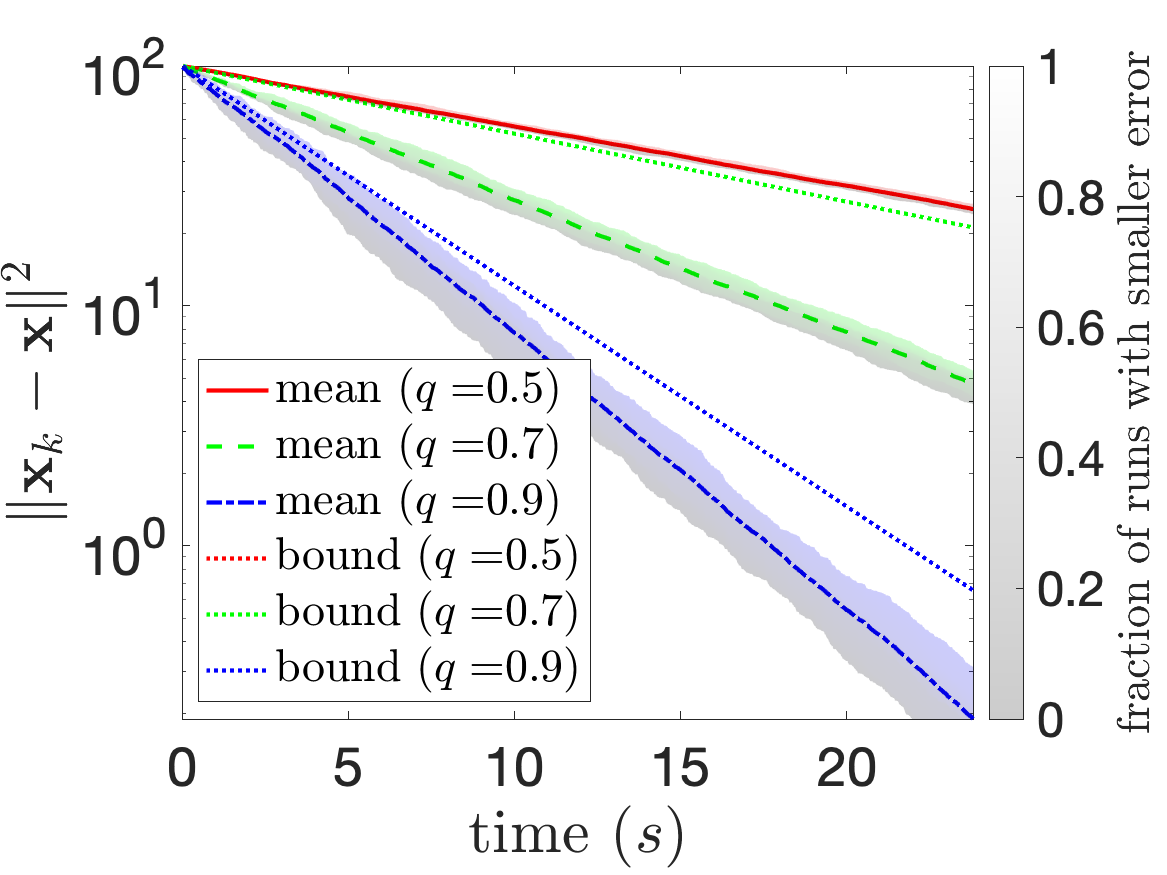}\hfill%
    \includegraphics[width=0.48\columnwidth]{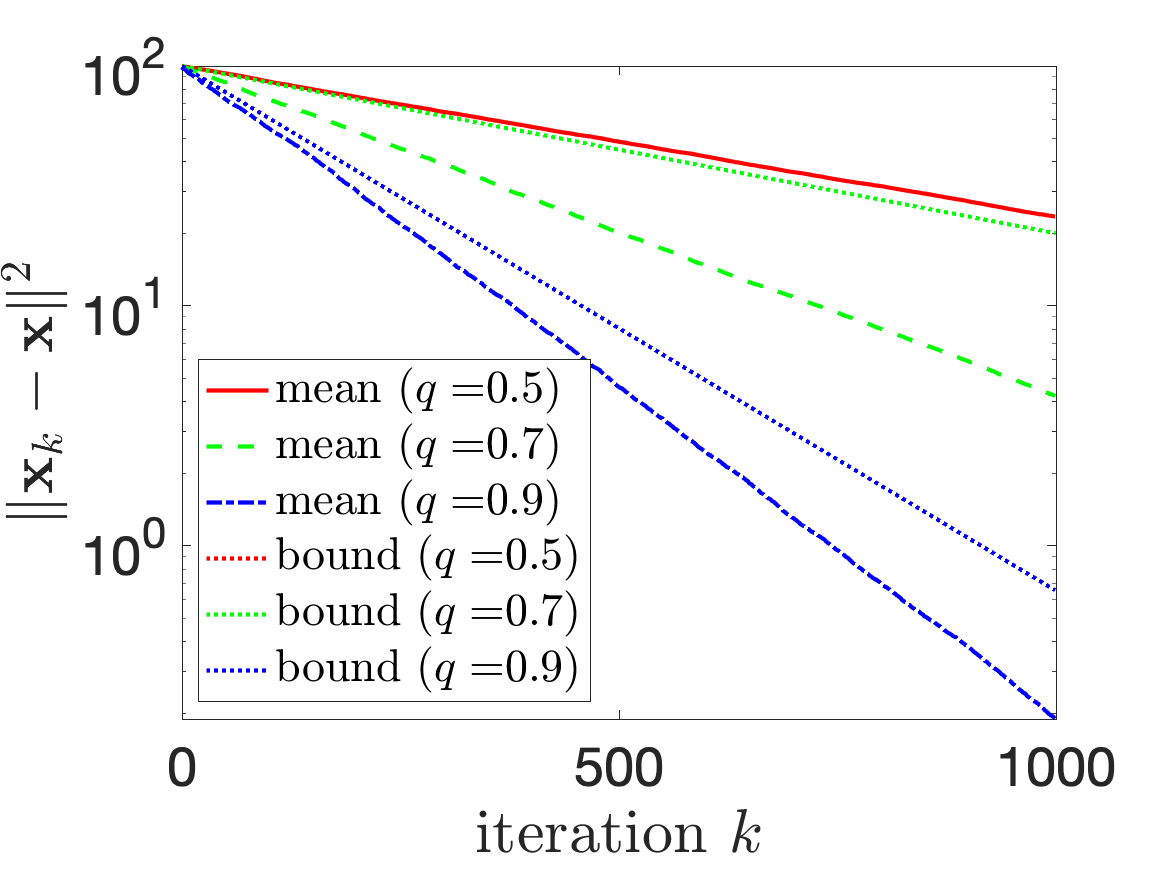}

    \includegraphics[width=0.48\columnwidth]{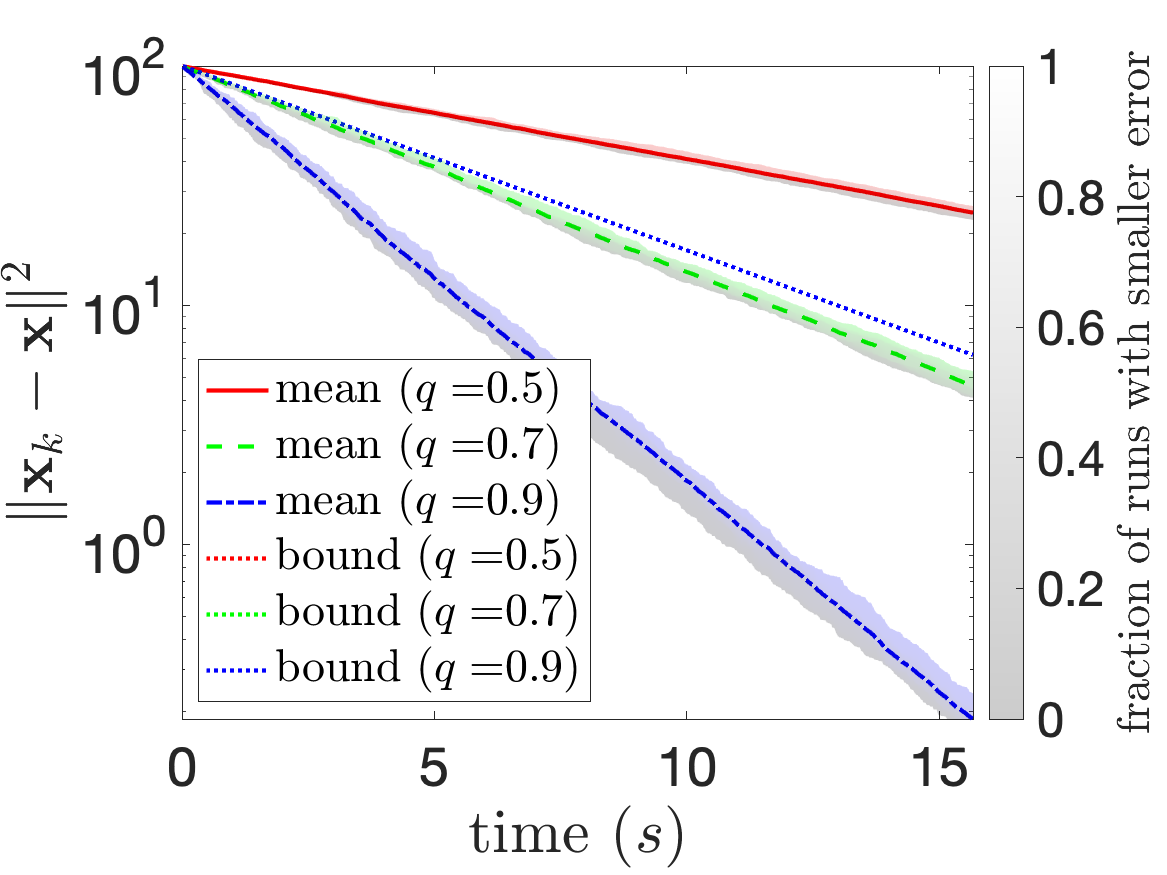}\hfill%
    \includegraphics[width=0.48\columnwidth]{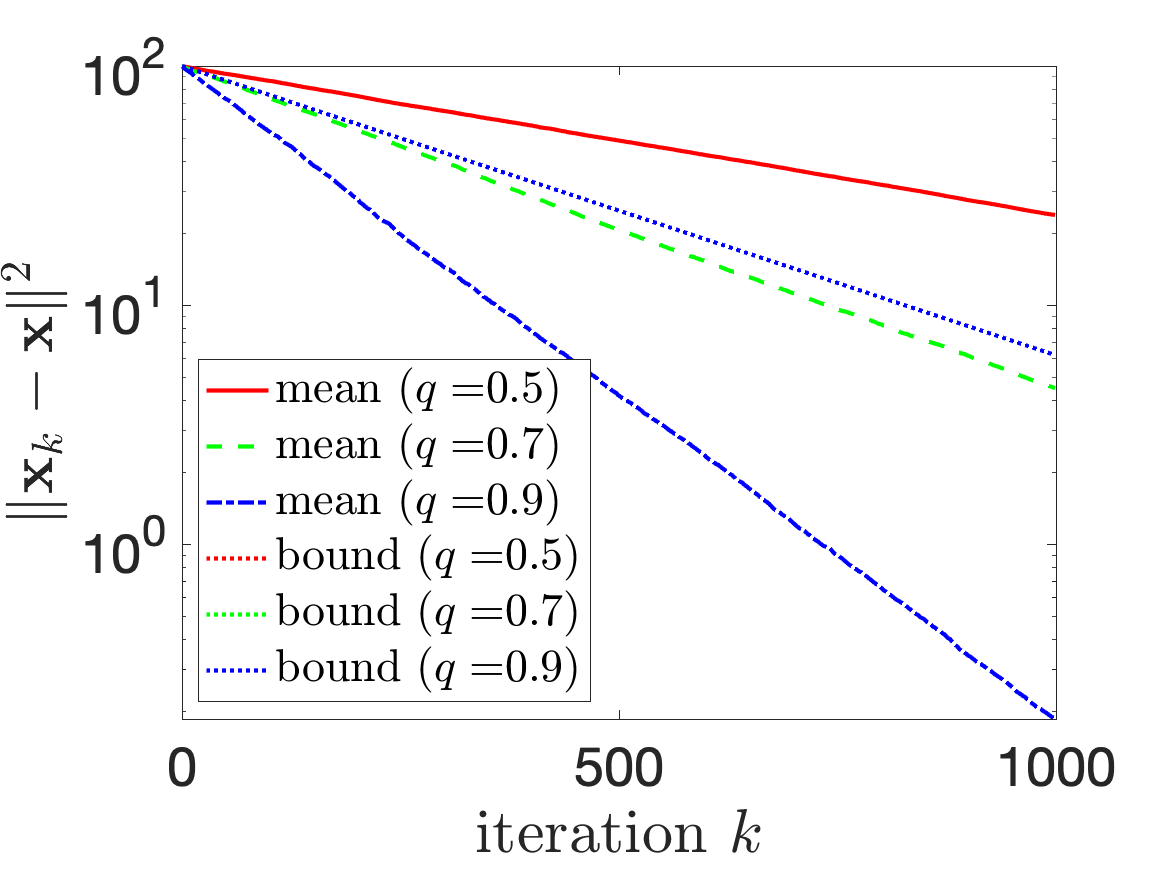}

    \includegraphics[width=0.48\columnwidth]{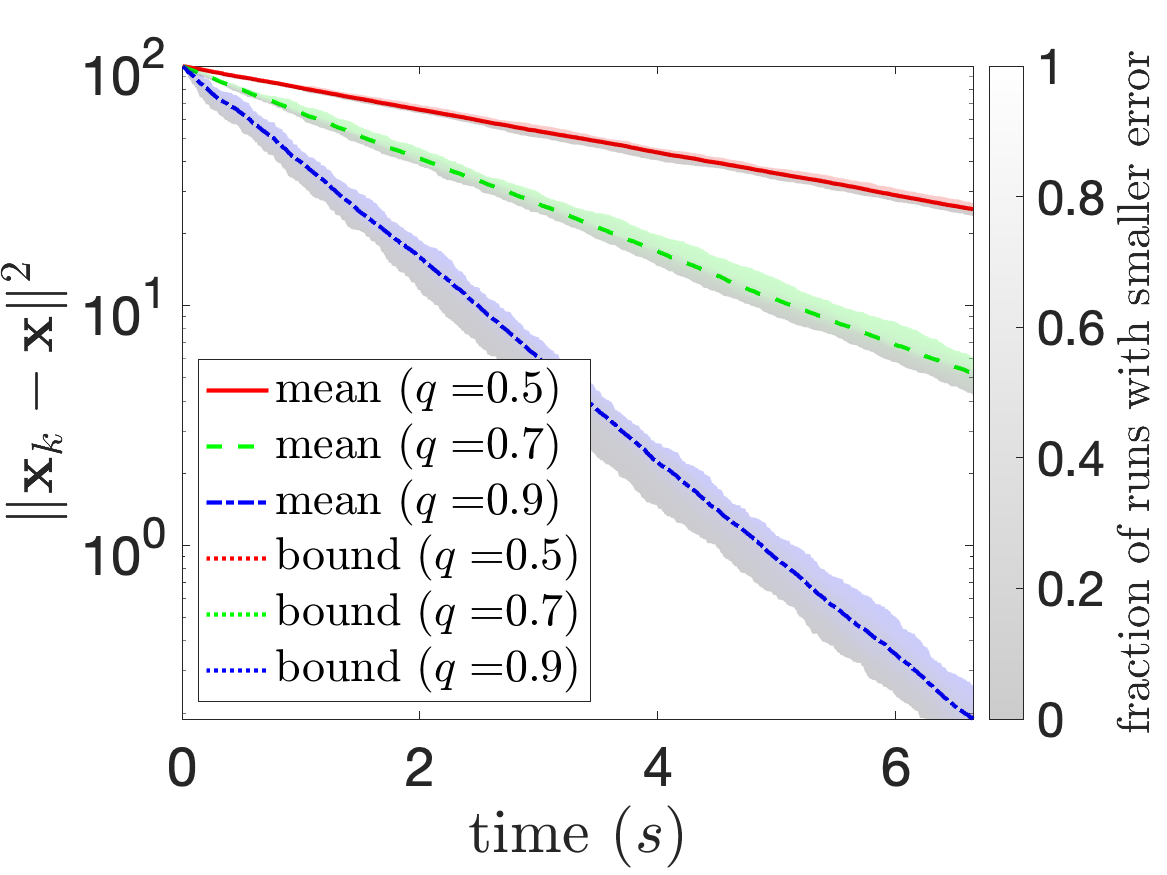}\hfill%
    \includegraphics[width=0.48\columnwidth]{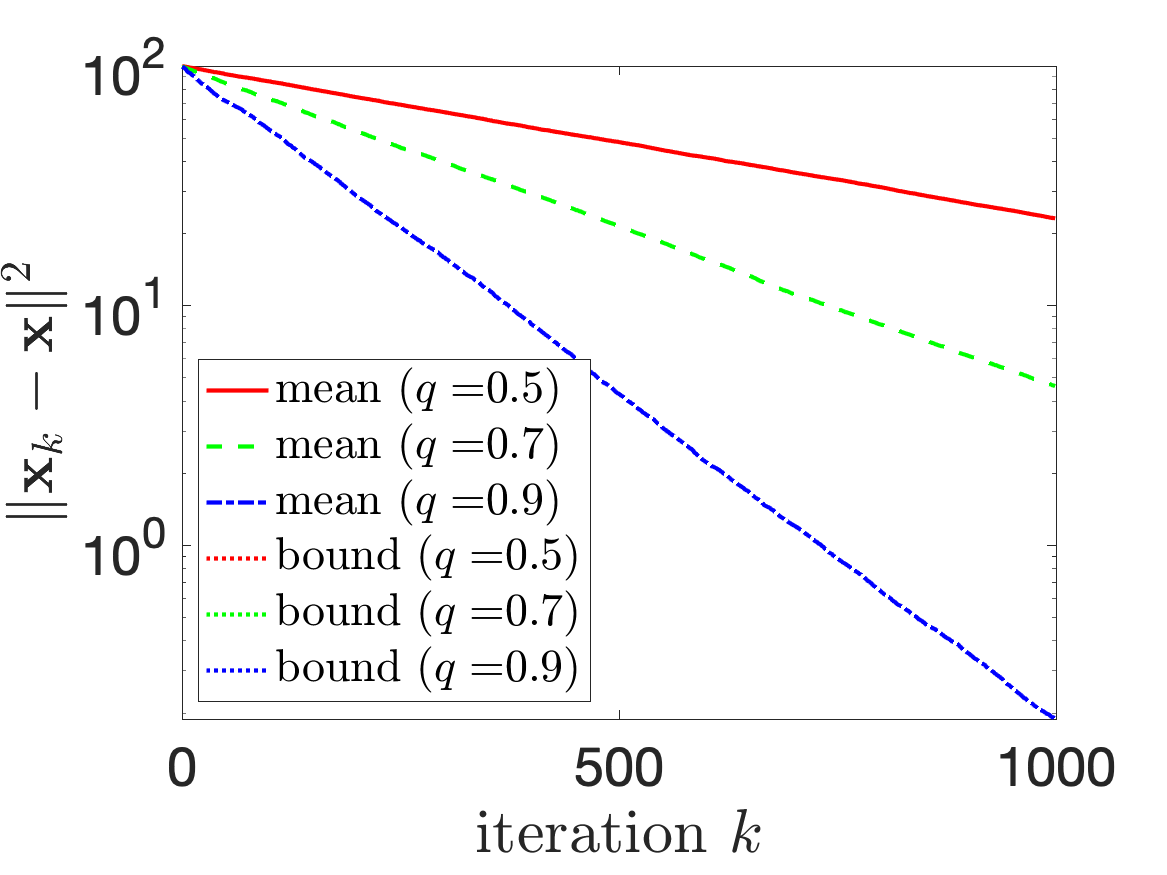}
    \caption{Empirical convergence with respect to time (left) and iterations (right) for ten trials of sQRK with $\alpha = 1$ (top), $\alpha = 0.5$ (middle), and $\alpha = 0.15$ (bottom) and $q = 0.5, 0.7,$ and $0.9$ on system defined by row-normalized $A \in \mathbb{R}^{50000 \times 100}$ and $\hat{\ve{b}} \in \mathbb{R}^{50000}$ with $\lfloor \beta m \rfloor$ corrupted entries where $\beta = 10^{-3}$.  We additionally plot the bounds provided by Theorem~\ref{thm:main} in dotted lines if they decrease.}\label{fig:convergence_qs_1e-03}
\end{figure}

In Figure~\ref{fig:convergence_qs_1e-03}, we plot the empirical convergence of ten trials of 1000 iterations of sQRK with $\beta = 10^{-4}$, $\alpha = 1$ (top), $\alpha = 0.5$ (middle), and $\alpha = 0.15$ (bottom), and $q = 0.5, 0.7,$ and $0.9$.  We plot the empirical convergence of ten independent trials of 1000 iterations of sQRK with respect to wall clock time on the left (brighter color of cloud indicates more trial errors were below) and the mean error for each $\alpha$. We plot the mean error for each $\alpha$ with respect to iteration $k$ on the right.  We plot the bounds given by Theorem~\ref{thm:main} that were decreasing.

\begin{figure}
    \includegraphics[width=0.48\columnwidth]{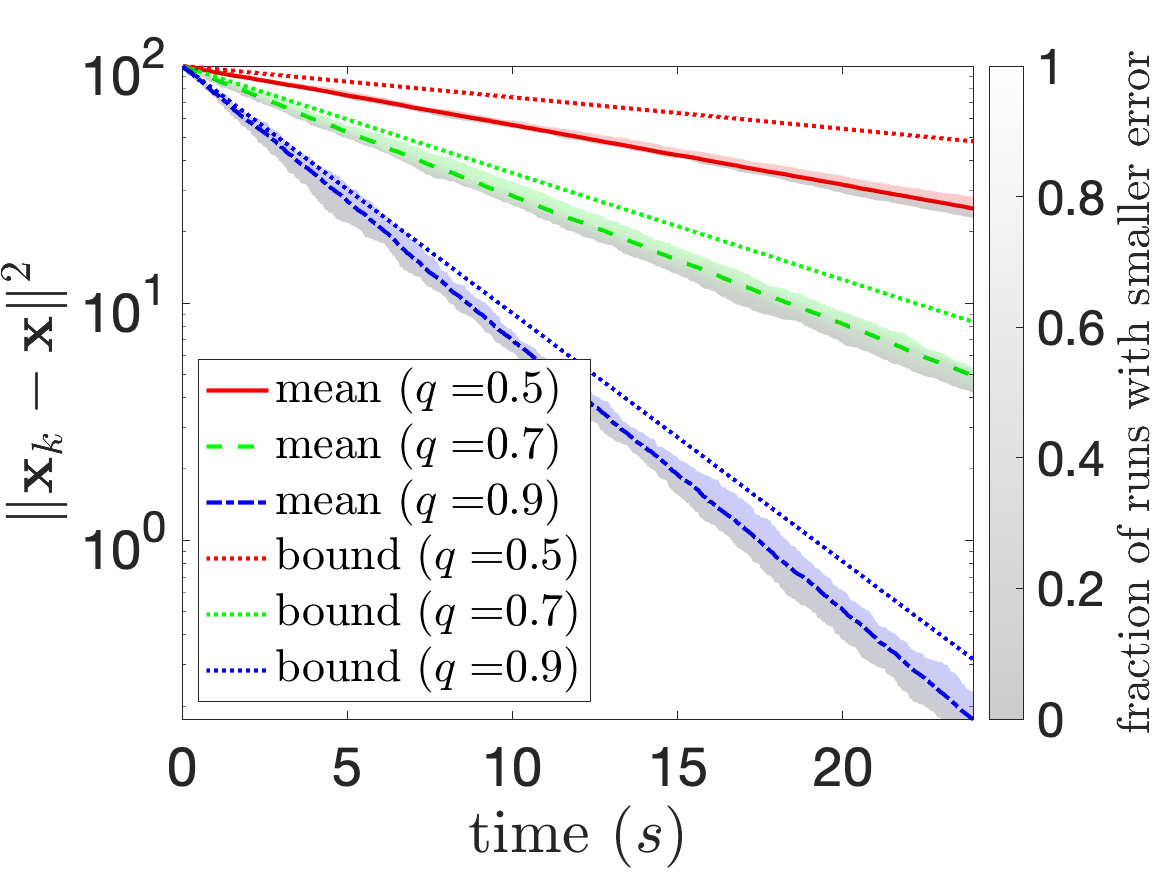}\hfill%
    \includegraphics[width=0.48\columnwidth]{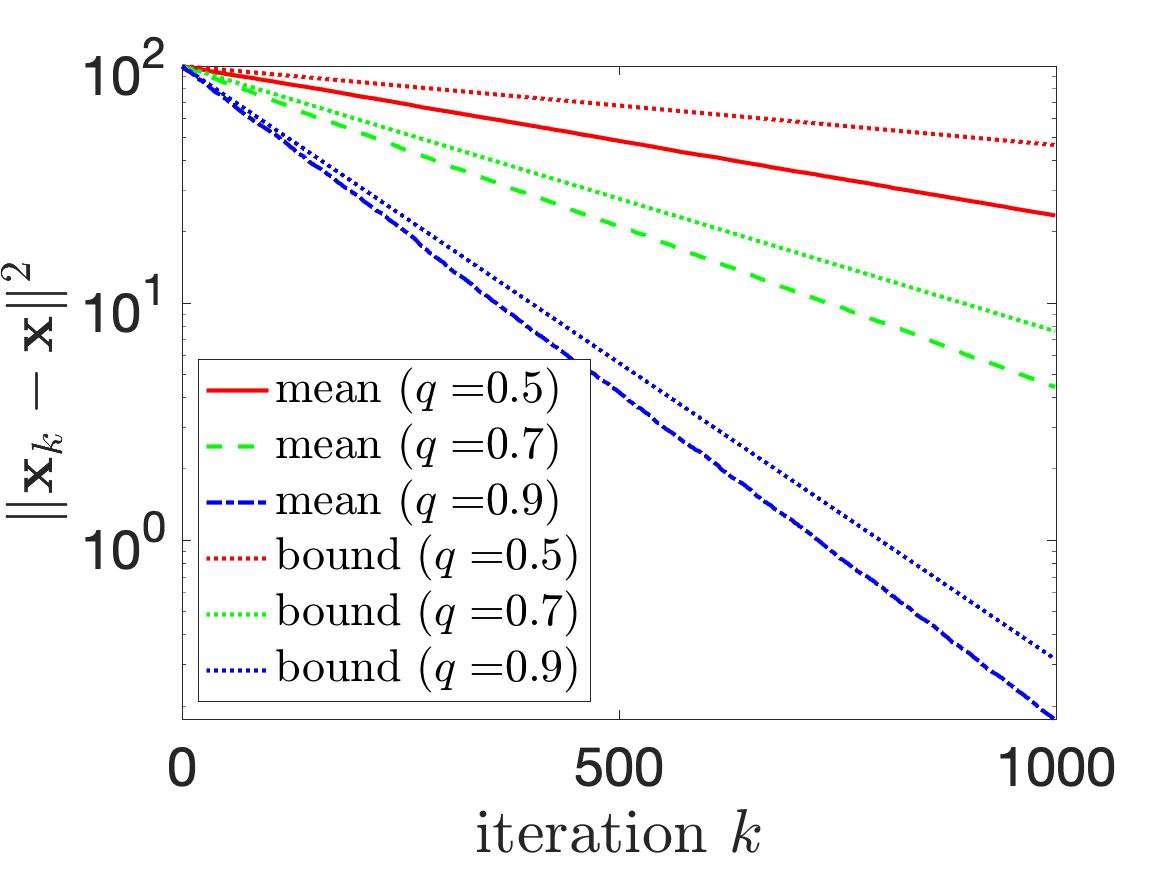}

    \includegraphics[width=0.48\columnwidth]{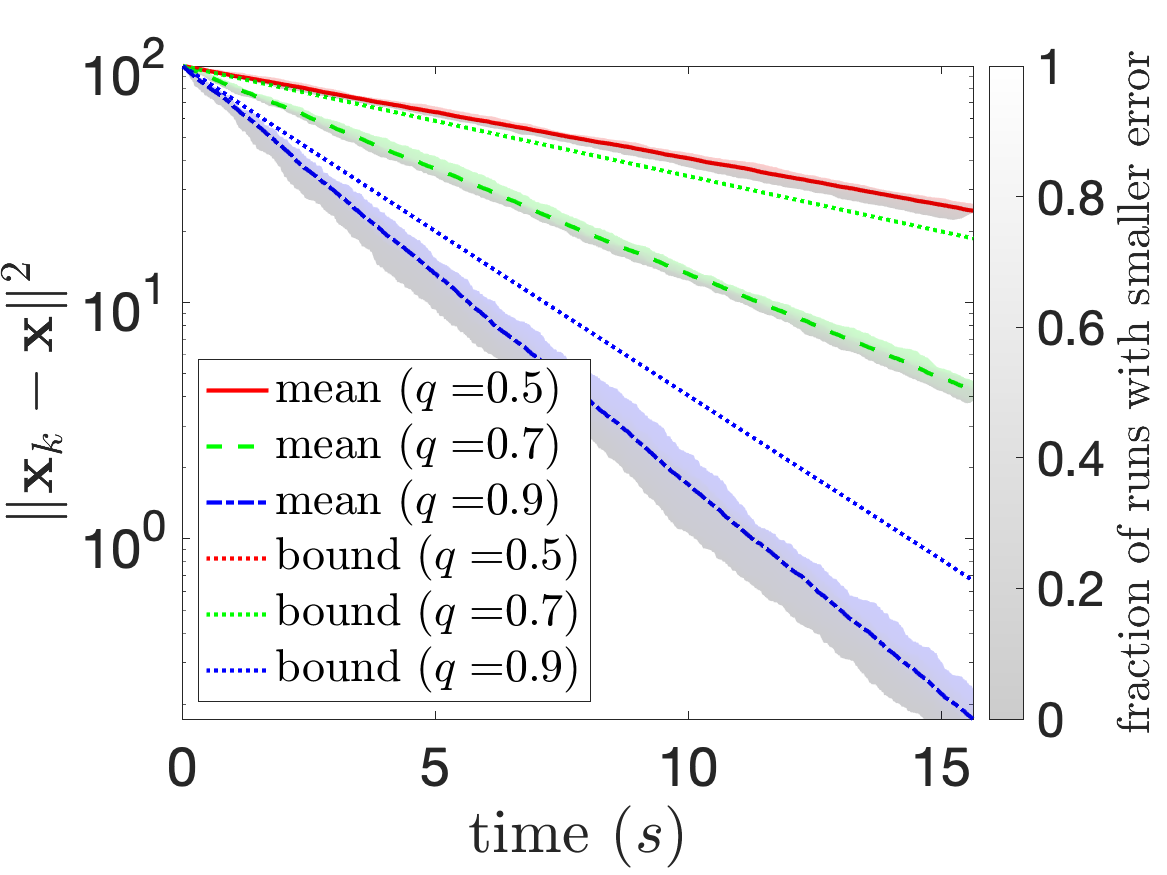}\hfill%
    \includegraphics[width=0.48\columnwidth]{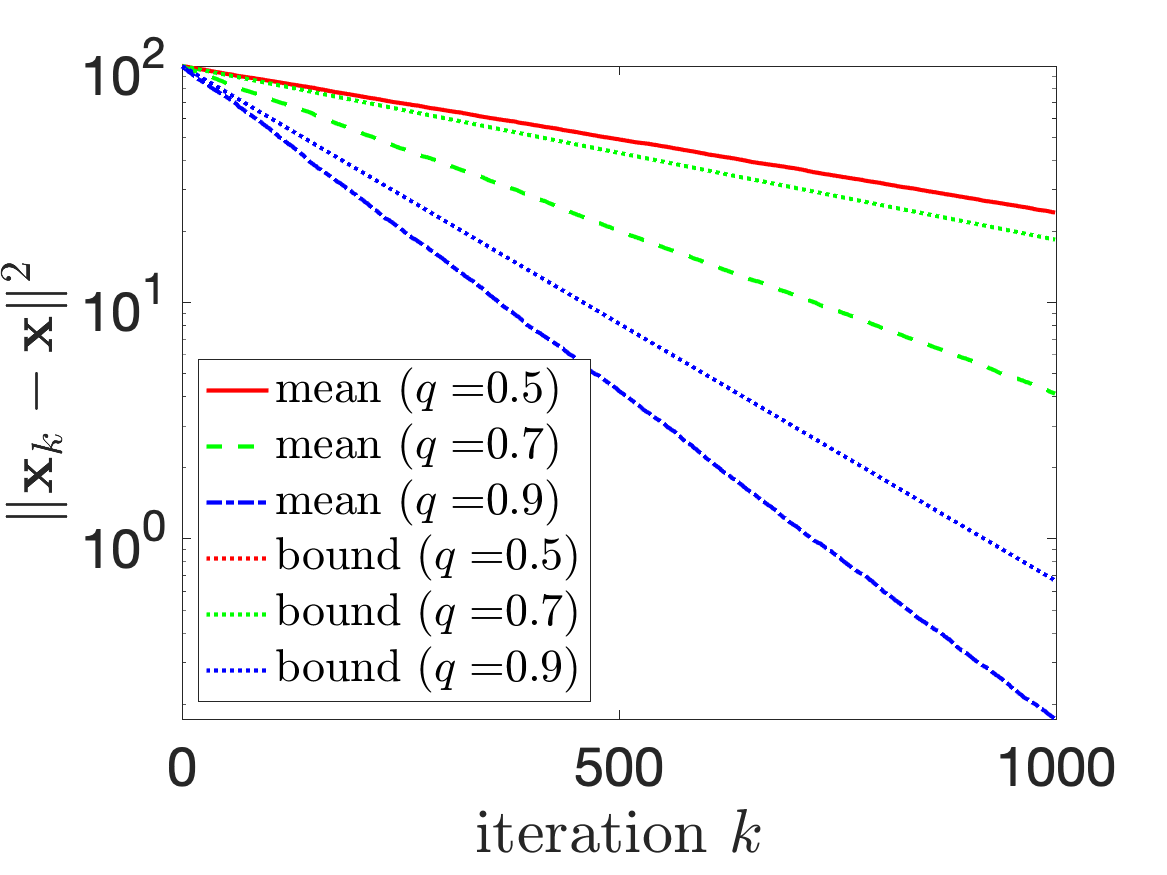}

    \includegraphics[width=0.48\columnwidth]{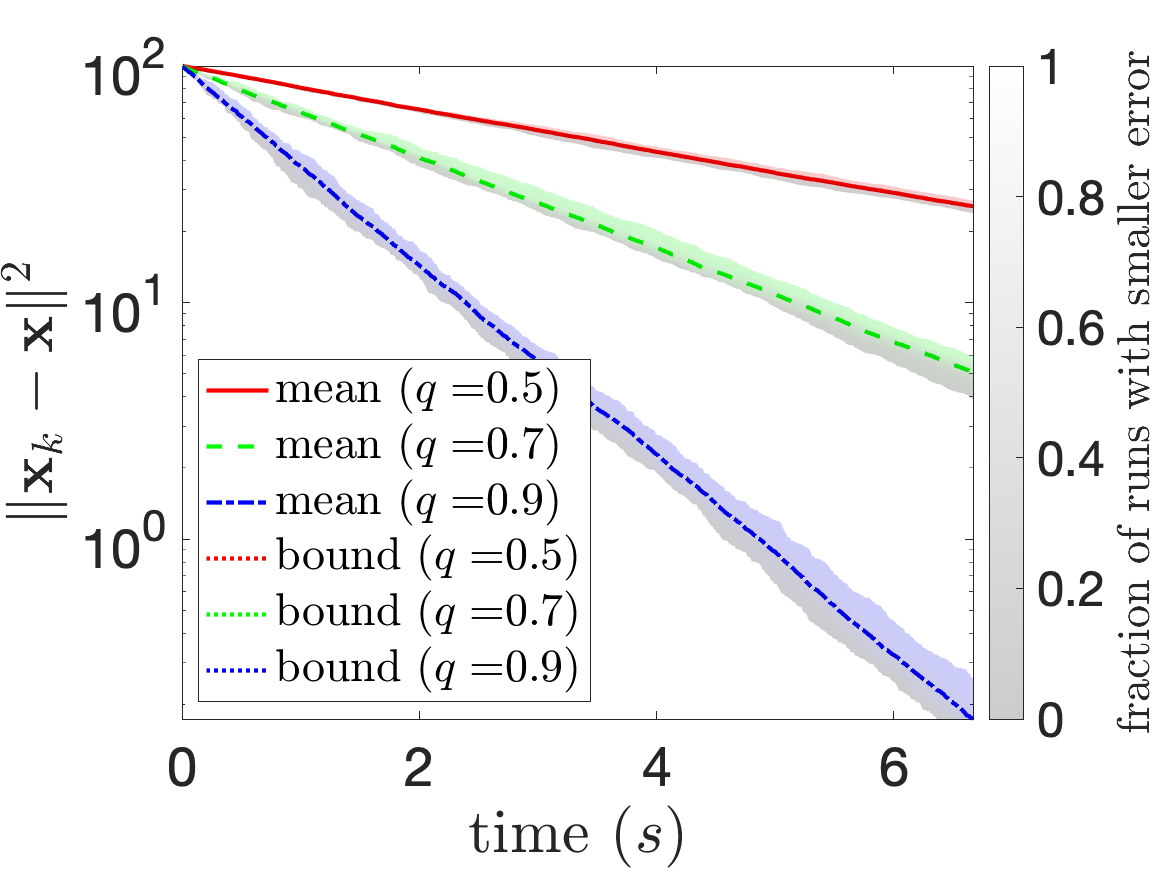}\hfill%
    \includegraphics[width=0.48\columnwidth]{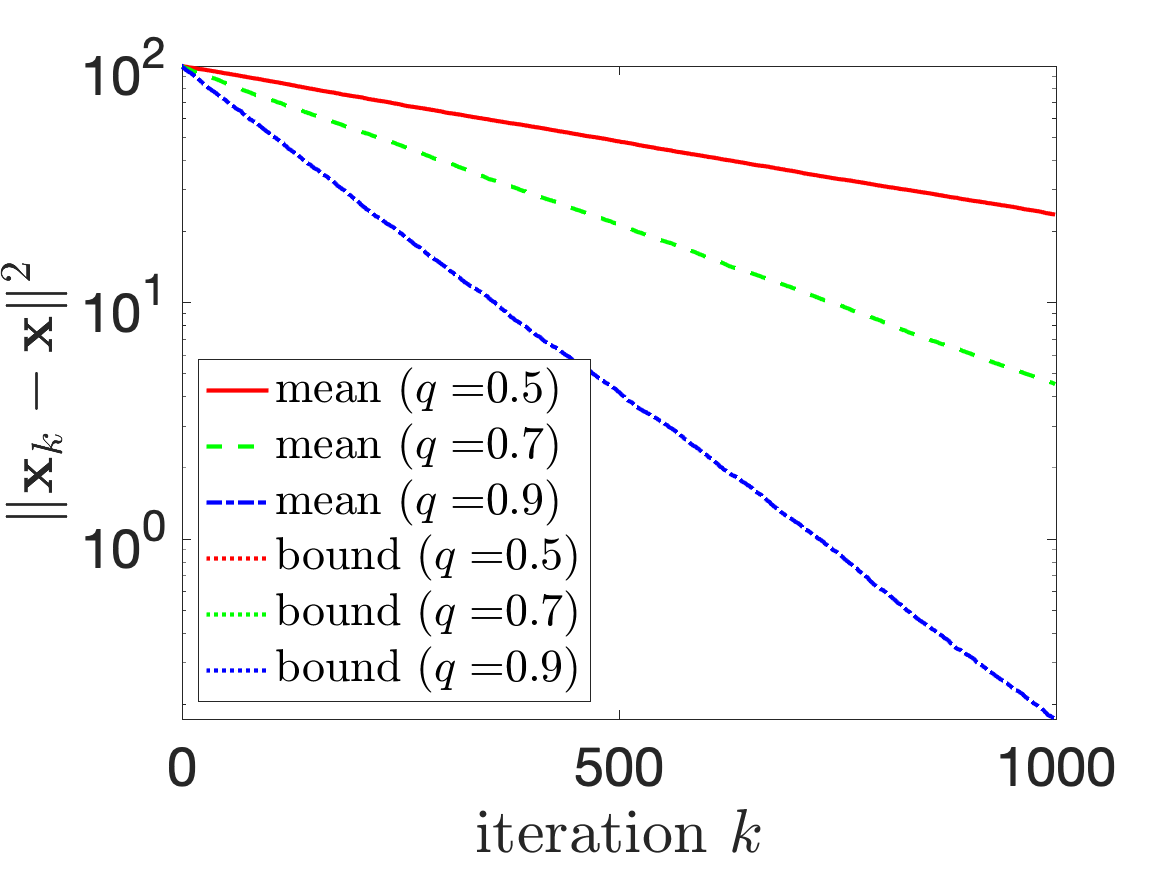}
    \caption{Empirical convergence with respect to time (left) and iterations (right) for ten trials of sQRK with $\alpha = 1$ (top), $\alpha = 0.5$ (middle), and $\alpha = 0.15$ (bottom) and $q = 0.5, 0.7,$ and $0.9$ on system defined by row-normalized $A \in \mathbb{R}^{50000 \times 100}$ and $\hat{\ve{b}} \in \mathbb{R}^{50000}$ with $\lfloor \beta m \rfloor$ corrupted entries where $\beta = 10^{-4}$.  We additionally plot the bounds provided by Theorem~\ref{thm:main} in dotted lines if they decrease.}\label{fig:convergence_qs_1e-04}
\end{figure}

\begin{figure*}[th!]
\centering
    \includegraphics[width=.3\textwidth]{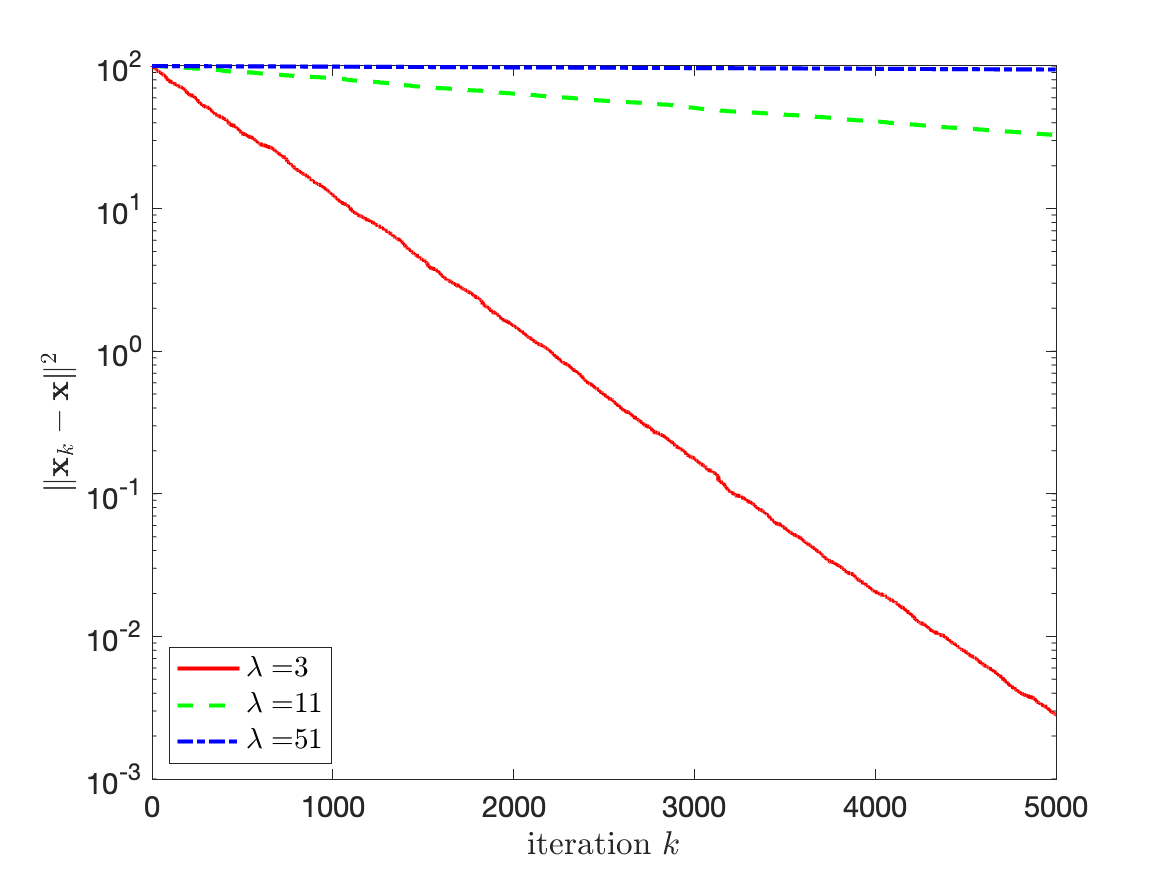}
    \includegraphics[width=.3\textwidth]{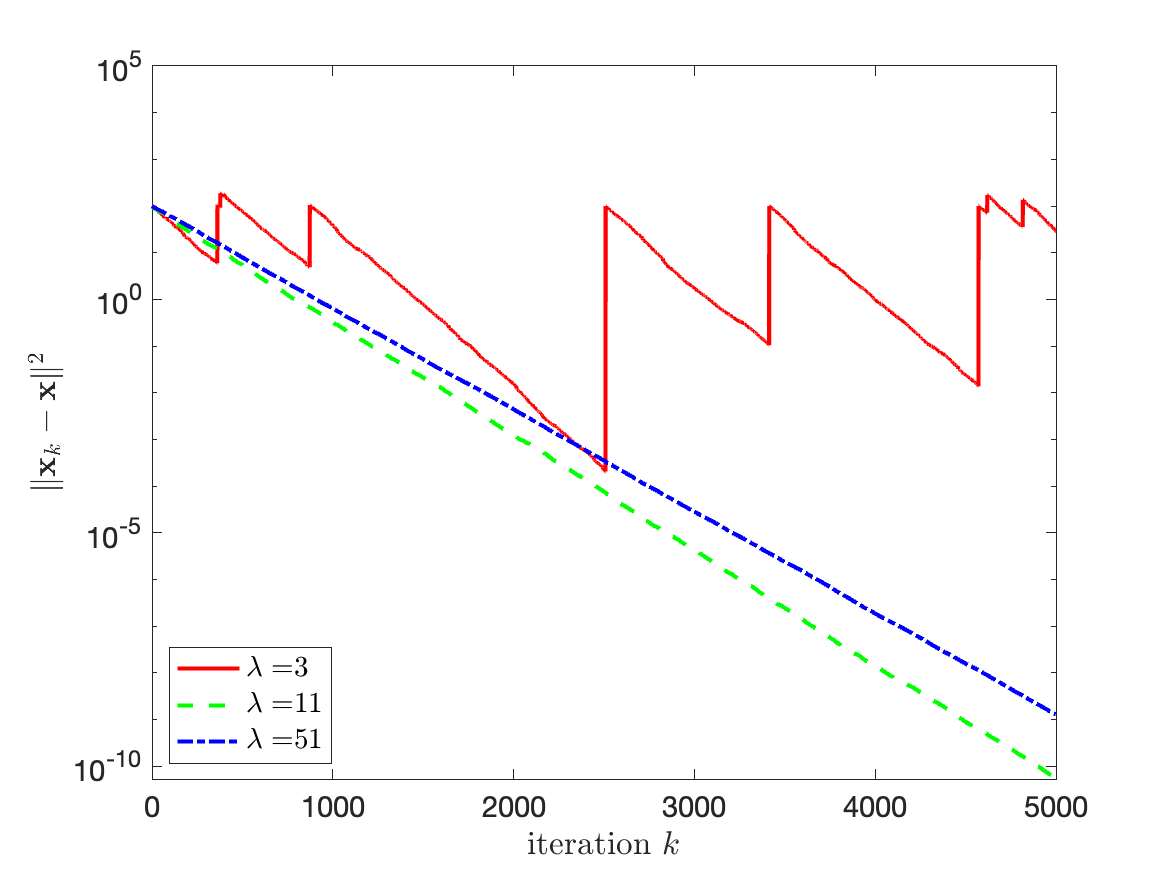}
    \includegraphics[width=.3\textwidth]{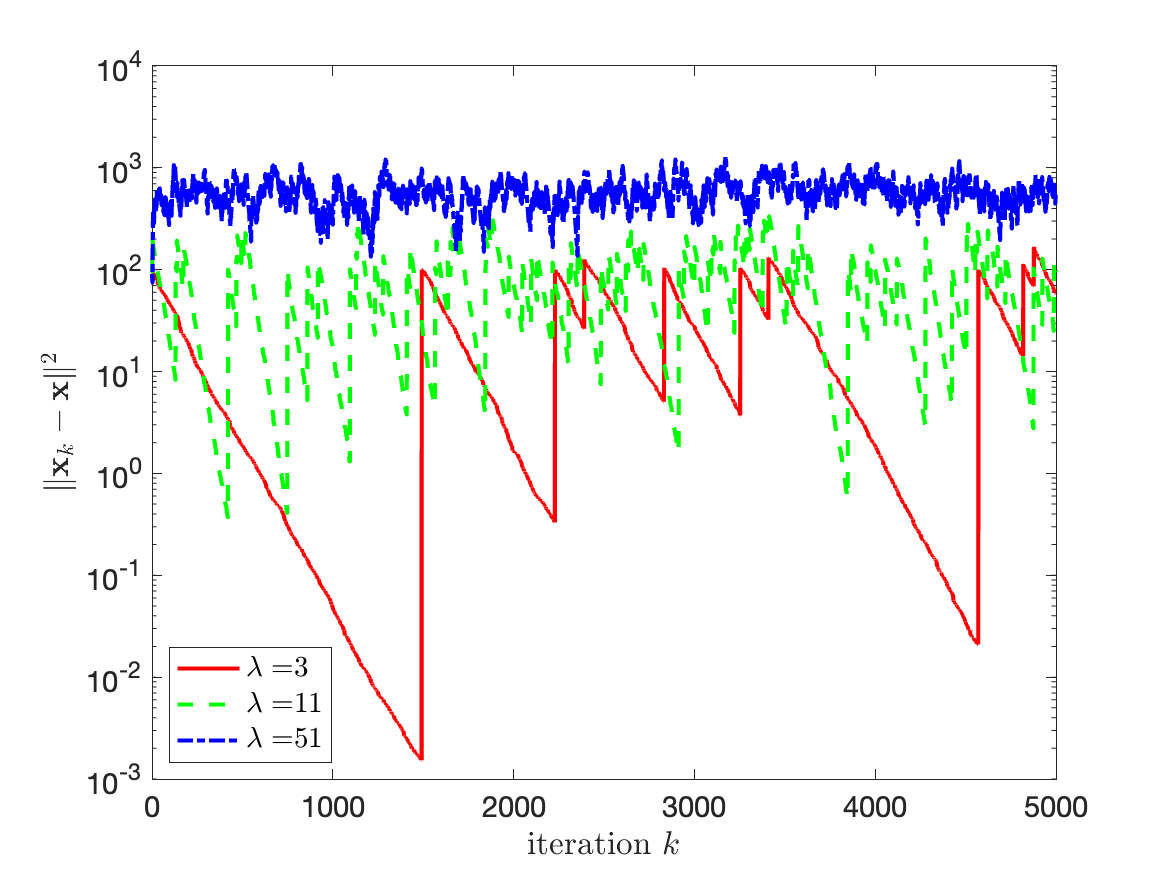}
    \caption{We present the performance of Algorithm~\ref{alg:sqrk-small} using $\lambda= \{3, 11, 51 \}$ when $ m = 50,000$, $n=100$, $\beta = 0.02$ for (left) $q = \frac{1}{\lambda}$, which corresponding to picking the smallest residual, (center) $q = \frac{1}{2}$, which corresponds to picking the median residual, and (right) $q = \frac{\lambda - 1}{\lambda}$, which corresponds to choosing the second largest residual.}\label{fig:smallsample}
\end{figure*}

\subsection{Small Samples}\label{small-samples}
The previous section considers subsets of the size on the order of $O(m)$. While it makes the quantile learning stage $\alpha^{-1}$ times faster, it still presents significant computation overhead at each step compared to the classical RK method (which steps scale linearly with $n$ and do not depend on $m \gg n$). Another approach is to approximate the quantile value from a considerably smaller, $O(1)$-size random sample $\lambda \ll m$ as outlined in Algorithm~\ref{alg:sqrk-small} below.
\begin{algorithm}
	\caption{Small Sample Quantile RK}\label{alg:sqrk-small}
	\begin{algorithmic}[1]
		\Procedure{ssQRK}{$\ve{A},\hat{\ve{b}}$, $q$, $\lambda$, $N$}
		\State{$\ve{x}_1 = \ve{0}$}
		\For{$k = 1, \ldots, N$}
		\State{Sample $\tau_k \subset [m]$ uniformly such that $|\tau_k| = \lceil \lambda \rceil $}
		\State{$\gamma_k = \text{$q$-quant}\left(\{| \langle \ve{a}_j,  \ve{x}_{k-1} \rangle - \hat{b}_j | \}_{j \in \tau_k }\right) $}
            \State{$i = \{i \in \tau_k : | \langle \ve{a}_i,  \ve{x}_{k-1} \rangle - \hat{b}_i | = \gamma_k \}$}
		\State{$\ve{x}_{k+1} = \ve{x}_{k} + (\hat{b}_i - \langle \ve{a}_i,  \ve{x}_{k} \rangle ) \ve{a}_i^T$ }
		\EndFor{}
		\Return{$\ve{x}_N$}
		\EndProcedure
	\end{algorithmic}
\end{algorithm}

We first note that this approach cannot work with completely arbitrary corruption. If the sample size is less than the number of corruptions, there is a nonzero probability that all the equations in the sample are corrupted and arbitrarily far from the true solution; thus, any choice of the next equation can ``undo" the earlier approach towards the solution. However, in many applications, it is natural to assume that some predetermined constant $ C$ bounds the maximal size of the corruption.

In such cases, the following heuristic can help quantify the convergence of the method. Let
$$
\gamma := \text{$q'$-quant}\left(\{| \langle \ve{a}_j,  \ve{x}_{k-1} \rangle - \hat{b}_j | \}_{j \in [m] }\right),
$$ be the $q'$-quantile of the full residual. %\am{What do you mean by ``true" quantile? Quantile using residuals from only noncorrupt rows or quantile using all rows (not just subsample)?}
Consider the following events:
\begin{itemize}
    \item $\mathcal{E}_1$ = selected equation $i$ is corrupted and has residual that is larger than $\gamma$
    \item $\mathcal{E}_2$ = selected equation $i$ is corrupted and has residual that is smaller than $\gamma$
    \item $\mathcal{E}_3$ = selected equation $i$ is uncorrupted.
\end{itemize}
For simplicity, let us consider the case of $q = 1/2$. We can decompose the expected error conditioned on these events:
\begin{align*}
    \mathbb{E} \|\ve{x}_{k+1} - \ve{x}\|^2 &= \mathbb{P}(i \in \mathcal{E}_1) \mathbb{E}_{i \in \mathcal{E}_1} \|\ve{x}_{k+1} - \ve{x}\|^2 \\
    & + \mathbb{P}(i \in \mathcal{E}_2) \mathbb{E}_{i \in \mathcal{E}_2} \|\ve{x}_{k+1} - \ve{x}\|^2 \\
    & + \mathbb{P}(i \in \mathcal{E}_3) \mathbb{E}_{i \in \mathcal{E}_3} \|\ve{x}_{k+1} - \ve{x}\|^2.
    \end{align*}
As in the proof of Theorem~\ref{thm:main}, we expect that $\mathbb{P}(i \in \mathcal{E}_3) \mathbb{E}_{i \in \mathcal{E}_3} \|\ve{x}_{k+1} - \ve{x}\|^2 < r'\|\ve{x}_{k} - \ve{x}\|^2$ and that the other two terms are small enough.
Indeed, the expected approximation error for event $\mathcal{E}_3$ depends on $q$. If $q \sim 0.5$, the bound on the error should be similar to that of RK. For larger $q$, we expect the contraction to be even stronger (i.e., $r'$ smaller; see similar works on the Sampling Kaczmarz-Motzkin method \cite{haddock2021greed, de2017sampling}). This suggests, as demonstrated in Figure~\ref{fig:smallsample}, that taking $q$ not too small and not too large is crucial for the successful performance of the method.

The first event has very small probability \begin{align*}
    \mathbb{P}(i \in \mathcal{E}_1) &\le \mathbb{P}(\text{ chosen residual is larger than $\gamma$ }) \\ &\le (1 - q')^{\lambda/2},
\end{align*} %\am{ what is $\mathbb{P}(\text{ large residual })$?}
and $\|\ve{x}_{k+1} - \ve{x}_k \|$ is bounded by $C$. The expected approximation error in event $\mathcal{E}_2$ can be computed like in Lemma~\ref{lem:corruptProjection} with $q = q'$, and $\alpha = 1$.  The main difficulty is estimating the probability that $q$-th residual in a random sample is (un)corrupted, which depends on the overall distribution of the corrupted residuals in the sample. If the corrupted equations tend to have larger residuals than uncorrupted ones, we will have $\mathbb{P}(i \in \mathcal{E}_2) < \beta$ and $\mathbb{P}(i \in \mathcal{E}_3) >  1- \beta$.
In Figure~\ref{fig:smallsample}, we give empirical evidence of the convergence of the small sample Quantile RK method. We note that picking large $q$ or too small $\lambda$ results in spiking behaviour (right and middle figures), and picking too small $q$ results in very slow convergence that can be easily overtaken by a rare event of facing a large corruption (left figure). However, $q \sim 0.5$ and $\lambda = 11 \ll 7500$ (that was the smallest sample size from Figures 2-7) demonstrate successful convergence as suggested in the discussion above.
%has \am{$\|\ve{x}_{k+1} - \ve{x}_k \|$ is bounded by $Q$?} the small step size bounded as above and $\mathbb{P}(\mathcal{E}_2) \le \mathbb{P}(\text{ corrupted equation })$.

\section{Conclusion and Future Directions}
This paper considers variants of the Quantile Randomized Kaczmarz method for linear systems with corruptions which provide a computational advantage in the quantile estimation stage. We provide theoretical and empirical convergence guarantees for the case when the quantile is estimated from the sample of the size of a fraction of the total number of equations. We also consider empirically the case when the quantile is estimated from a very small sample of constant size (in particular, when it is smaller than the total number of corruptions in the system). Some interesting future directions of this work include providing theoretical analysis for the small sample size, particularly getting guidance for choosing the optimal sample size $\lambda$ and quantile $q$. Further, extending the more efficient small sample quantile ideas to the nonlinear problems with corruptions is another valuable future direction.

\section*{Acknowledgements}
We thank Jaime Pacheco (HMC) and Nestor Coria (HMC) for their assistance with the clarification of proofs and Max Collins (HMC) for the code, which helped to generate the figures in this work.

\bibliographystyle{ieeetr}
\bibliography{main}
\end{document}